\documentclass[a4,12pt]{article}

\usepackage[english]{babel}

\usepackage[letterpaper,top=2cm,bottom=2.5cm,left=3cm,right=3cm,marginparwidth=1.75cm]{geometry}

\usepackage{amsmath,amssymb,amsfonts,amsthm}
\usepackage[dvipdfmx]{graphicx}
\usepackage[colorlinks=true, allcolors=blue]{hyperref}
\usepackage{mathtools}
\usepackage{enumerate}
\usepackage{cite}

\newtheorem{thm}{Theorem}[section]
\newtheorem{cor}[thm]{Corollary}

\newtheorem{lem}[thm]{Lemma}
\newtheorem{rem}[thm]{Remark}

\newtheorem{definition}[thm]{Definition}

\numberwithin{equation}{section}
\numberwithin{thm}{section}

\title{Reservoir computing with the Kuramoto model}

\author{
    Hayato Chiba\thanks{Advanced Institute for Materials Research, Tohoku University, Japan.} \and
    Koichi Taniguchi\thanks{ Department of Mathematical and Systems Engineering, Faculty of Engineering, Shizuoka University, Japan.}  \and
    Takuma Sumi\thanks{Hotchkiss Brain Institute, Cumming School of Medicine, University of Calgary, Canada}
}

\date{Last update: 20 Oct 2025}

\begin{document}
\maketitle

\begin{abstract}

Reservoir computing aims to achieve high-performance and low-cost machine learning with a dynamical system as a reservoir. 
However, in general, there are almost no theoretical guidelines for its high-performance or optimality. 
Therefore, this paper aims to propose the new concept {\it the edge of bifurcation}
for designing a high-performance reservoir, and provide a mathematical justification for it.
This concept is a generalization of the famous criterion {\it the edge of chaos}.
For this purpose, 
this paper focuses on the reservoir computing with the Kuramoto model and theoretically reveals its approximation ability.
The main result provides an explicit expression of the dynamics  
of the Kuramoto reservoir by using the order parameters. 
Thus, the output of the reservoir computing is expressed as a linear combination of the order parameters.
As a corollary, sufficient conditions on hyperparameters are obtained so that the set of the order parameters gives the complete basis of the Lebesgue space.
This implies that the Kuramoto reservoir has a universal approximation property.
Furthermore, the edge of bifurcation 
is also discussed from the viewpoint of its approximation ability. 
It is numerically demonstrated by prediction tasks.

\end{abstract}

\section{Introduction}

Reservoir computing is one of the methods of recurrent machine learning with dynamical systems for computations on time series data 
(see \cite{jaeger2001,natschlager2002}) 
and has various applications such as time series prediction, 
system identification, signal generation and signal transformation (see \cite{Luk_2009,Yan2024}). 
The architecture of reservoir computing consists of two components, namely, reservoir and readout. 
The reservoir is a fixed recurrent nonlinear map (dynamical system) which appropriately maps input data to a high dimensional space. 
The readout is a time-independent linear map (matrix), and only the readout is learned to minimize the error between its output and the desired result. 
Since the learning of reservoir computing is done only by simple algorithms such as linear regression, 
reservoir computing has a distinctive advantage to make it possible to learn at low cost and to optimize instantly. 
On the other hand, its ability to achieve high performance is strongly dependent on the choice of a reservoir. 

The approximation ability of various reservoir computing models have been theoretically studied (see e.g. \cite{Goton2023,GGO2023,GO2018,Hart2024}). 
These results demonstrate that the reservoir computing can possess sufficient approximation ability
under appropriate choices of reservoir dynamics and parameters. 
This highlights the importance of finding practical or theoretical criteria for designing a high-performance reservoir.
One of the well-known criteria is
so-called the edge of chaos (or edge of stability) \cite{packard1988,langton1990,crutchfield1990}. 

The edge of chaos states that reservoir computing has its greatest computational capacity near the phase transition point
from an ordered state to a chaotic state. 
The effectiveness of the edge of chaos has been experimentally confirmed in various settings
(see e.g. \cite{carroll2020,tanaka2019,Yam_2016} for the details of the edge of chaos). 
In \cite{toyoizumi2011}, the edge of chaos is 
theoretically studied for recurrent neural networks.
However, to the best of our knowledge, rigorous theoretical analyses and mathematical guarantees for the edge of chaos are still very limited.

Recently, 
it has been reported that the performance of reservoir computing can be improved without the edge of chaos. 
Specifically, Goto, Nakajima and Notsu
\cite{goto2021twin} studied reservoir computing with the Navier-Stokes equation as a reservoir and showed that 
the optimal computational performance is achieved near the critical Reynolds number for  the Hopf bifurcation. 
In light of this recent study,
we propose \textit{the edge of bifurcation}, which is a generalization of the edge of chaos, and it means that the performance gets better 
when the dynamical system undergoes some bifurcations (e.g. Hopf bifurcation).

In this paper, we tackle the theoretical analysis of the edge of bifurcation and aim to provide mathematical insights in this topic.
For this purpose, we focus on the Kuramoto model as a reservoir,
which is the most typical mathematical model for synchronization phenomena. 
Recently, reservoir computing with the Kuramoto model has been studied numerically (see \cite{Wang2022,zuo2023self}), 
but its potential has not yet been fully clarified, and there is a lack of theoretical results on its properties.
The purpose of this paper is to theoretically clarify the approximation ability of  reservoir computing with the Kuramoto model 
and to provide mathematical insights on the edge of bifurcation for reservoir computing. 

The basic results of the Kuramoto model, which are the fundament of this paper, are summarized in Section \ref{sec:2}.
A summary of our results and contributions is as follows: 
\begin{itemize}
\item {\bf (Formulation of reservoir computing and numerical method)} 
The first contribution is to formulate the Kuramoto reservoir computing with the aid of a certain integral operator for mathematical analysis. 
More precisely, we propose reservoir computing by using the infinite-dimensional Kuramoto model as a reservoir,
which is a continuous limit ($N\to \infty$) of the Kuramoto model as the number $N$ of oscillators tends to infinity, 
and a readout is not a matrix but an integral operator (see Section \ref{sec:3}). 
Furthermore, our reservoir computing is easily implementable and its specific numerical calculation procedure is also given via the Fourier series 
(see Section \ref{sec:5}).

\item {\bf (Approximation ability of the Kuramoto reservoir)}
The second contribution is to theoretically clarify the approximation ability of reservoir computing with the Kuramoto model. 
We can obtain a useful representation of the Kuramoto reservoir computing via the Fourier series.  
More precisely, we assume that the integral kernel of the integral operator for the readout  belongs to $L^2(0,2\pi)$.
Then the output function can be represented by a linear combination of the $n$-th order parameters 
of a solution of the infinite-dimensional Kuramoto model (see \eqref{key_representation}). 
Therefore, if the family of $n$-th order parameters constructs a complete basis of $L^2_{\mathrm{per}}$ (the set of periodic $L^2$ functions), 
the Kuramoto reservoir has an enough approximation ability. 
Our main theorem provides an explicit expression of the $n$-th order parameters  (Theorem \ref{thm:main}) .
As a corollary, a sufficient condition for this family to be a complete basis of $L^2_{\mathrm{per}}$ (Corollary \ref{cor:main}) will be obtained. 
The precise statements and proofs are given in Section \ref{sec:4}.
In addition, we also provide the approximation theorems and error estimates for prediction tasks 
(Sections 5 and 6).
These results are numerically supported in Section 7. 

\item ({\bf Edge of bifurcation)} 
The third contribution is to provide some mathematical insights on the edge of bifurcation conjecture 
for reservoir computing with the Kuramoto model from a viewpoint of the approximation ability. 
A key parameter is the coupling strength $K$ of the Kuramoto model, which determines  
the occurrence of synchronization. 
There is a critical value $K_c$ such that a bifurcation occurs at $K=K_c$ under certain conditions (see \cite{Kur_1975,Kur_1984,chiba2015}); 
the de-synchronous state is stable when $K<K_c$ and the synchronous state appears when $K>K_c$.  
Hence, it is expected that this $K_c$ has a strong influence on the performance of the Kuramoto reservoir computing in terms of the edge of bifurcation. 
In this paper, we consider the edge of bifurcation in some settings based on our approximation theorem and the bifurcation theory of the Kuramoto model.
As a consequence, we can see that
the type of bifurcation and the shape of bifurcation diagram are important in the performance of Kuramoto reservoir computing. 
For example, it turns out that a bifurcation should be not a pitchfork bifurcation (steady to steady), but be a Hopf bifurcation (steady to periodic). 
Further, we will show that a sufficient condition for the family of the $n$-th order parameters 
to be the basis of  $L^2_{\mathrm{per}}$ is satisfied when $K$ is slightly larger than $K_c$, that guarantees the edge of bifurcation. 
See Section \ref{sec:4} for the details of these observations.
This result is also numerically supported in Section \ref{sec:5}.

\end{itemize}


\section{Kuramoto model}\label{sec:2}

The (finite-dimensional) Kuramoto model is the system of ordinary differential equations given by
\begin{equation}\label{KM}
    \frac{d \theta_i}{dt} = \omega_i + \frac{K}{N} \sum_{j=1}^N \sin (\theta_j - \theta_i),\quad i=1,\cdots ,N,
\end{equation}
where $i \in [N] := \{1,2,\ldots, N\}$, $\theta_i = \theta_i(t) \in [0,2\pi)$ denotes the phase of an $i$-th oscillator on a circle, 
$\omega_i \in \mathbb R$ is its natural frequency distributed according to some probability density function $g(\omega)$, 
$K>0$ is a coupling strength, and $N$ is the number of oscillators.
We normally assume that $g(\omega)$ is an even and unimodal function about its mean frequency $\Omega \in \mathbb R$: 
(even) $g(\Omega + \omega)=g(\Omega-\omega)$ for any $\omega\in \mathbb R$;
(unimodal) $g(\omega_1)>g(\omega_2)$ for $\Omega\le \omega_1 <\omega_2$ and $g(\omega_1)<g(\omega_2)$ for $\omega_1 <\omega_2\le \Omega$.
This system was proposed by Kuramoto \cite{Kur_1975} in order to investigate collective synchronization phenomena 
and was derived by means of the averaging method from coupled dynamical systems having limit cycles.
The order parameter is usually used to measure the degree of synchronization of the oscillators $\{\theta_i\}_{i=1}^N$ and is defined by 
\begin{equation}\label{Order}
r (t) := \frac{1}{N} \sum_{j=1}^N e^{i\theta_j (t)},
\end{equation}
where $i = \sqrt{-1}$. The order parameter gives the centroid of oscillators. 
If $|r|>0$, then the synchronization occurs, while if $|r|$ is nearly equal to zero, 
the oscillators are uniformly distributed (de-synchronization), see Fig. \ref{fig:1}. 

\begin{rem}
For the study of the Kuramoto model, we can assume without loss of generality that $\Omega = 0$.
Indeed, changing variables for \eqref{KM} as $\theta_i \mapsto \theta_i + \Omega t$ yields
\begin{equation}\label{KM2}
    \frac{d \theta_i}{dt} = \omega_i-\Omega + \frac{K}{N} \sum_{j=1}^N \sin (\theta_j - \theta_i),\quad i=1,\cdots ,N.
\end{equation}
Then, the mean value of $\{ \omega_i-\Omega \}_i$ is zero.
Nevertheless, we will see that for an application to a reservoir computing, nonzero $\Omega$ plays an important role (Corollary \ref{cor:main}).
\end{rem}

\begin{figure}[t]
\begin{center}
\includegraphics[width=100mm]{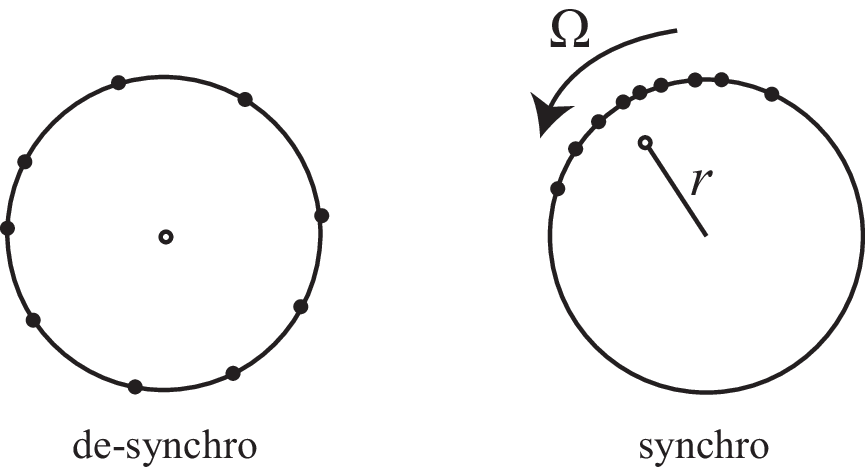}
\caption{Schematic representation of oscillators on a circle. 
The left represents a de-synchronous state with $|r|\approx 0$.  The right represents a synchronous state with $|r|\approx 1$.}
\label{fig:1}
\end{center}
\end{figure}

The coupling strength $K$ is the key bifurcation parameter that determines whether the dynamics becomes synchronous or de-synchronous, 
and it is important to obtain the bifurcation diagram of $r(t)$ for the study of the edge of bifurcation.
In this paper, we 
address the case where $N$ is large and consider the continuous limit of \eqref{KM} as $N\to\infty$, 
namely, the infinite-dimensional Kuramoto model given by 
\begin{equation}\label{CKM}
\begin{cases}
	\dfrac{\partial\rho}{\partial t} + \dfrac{\partial}{\partial \theta} (\rho v) =0,\\
	v =  \omega + K |r| \sin (\psi -\theta),\\
	\displaystyle r(t) = |r(t)| e^{i\psi(t)} := \int_{\mathbb R} \int_0^{2\pi} e^{i\theta}\rho(\theta, \omega, t) g(\omega)\, d\theta d\omega,
\end{cases}
\end{equation}
where $\rho = \rho(\theta, \omega,t)$ is the density of oscillators at a phase $\theta$ parameterized by natural frequency $\omega$ and time $t$, 
$v=v(\theta, \omega,t)$ is the drift velocity of the oscillators, and 
$r(t)$ is the complex order parameter of a solution $\rho$ with its argument $\psi$. 
For any $K>0$,
the trivial steady state solution of \eqref{CKM} exists 
(i.e. a completely incoherent state given by $\rho(\theta,\omega)=1/(2\pi)$ 
for any $\theta, \omega$) and it corresponds to the de-synchronous state $r \equiv 0$.
Kuramoto \cite{Kur_1984} conjectured that a bifurcation diagram of $r(t)$ is given as Fig. \ref{fig:2} (a):
\\

\noindent\textbf{Kuramoto conjecture}

Suppose that $N\to \infty$ and natural frequencies 
$\omega_i$'s are distributed according to an even and unimodal function $g(\omega )$ about its mean value $\Omega$ with $g''(\Omega) < 0$
(this is true for Fig. \ref{fig:2} (a)).
If $K<K_c:= 2/(\pi g(\Omega ))$, then $|r(t)| \equiv 0$ is asymptotically stable, while if $K>K_c$, 
the synchronous state emerges and $|r(t)|\equiv r(K)$ is asymptotically stable for some $r(K)>0$ given by
\begin{equation}\label{kuramoto}
r(K) = \sqrt{\frac{-16}{\pi K_c^4 g''(0)}}\sqrt{K-K_c} + O(K-K_c)
\end{equation}
(see \cite{strogatz2000} for Kuramoto's discussion).\\
  
The critical value $K_c$ is called the Kuramoto transition point.
Then, Chiba \cite{chiba2015} 
mathematically formulated this problem and proved the Kuramoto conjecture. 
See \cite{chiba2015} for the precise statement.


It is known that the type of bifurcation strongly depends on the assumption
for $g(\omega)$, especially for the sign of $g''(\Omega)$.
For example, when $g(\omega)$ is a uniform distribution (i.e. $g''(\Omega) = 0$),
the bifurcation diagram is given as Fig. \ref{fig:2} (b).
In this paper, as a reservoir, we consider both cases (a) $g$ is Gaussian (b) $g$ is a uniform distribution.
The important difference is that for a uniform distribution case, the order parameter
suddenly jumps to a certain value at the bifurcation point $K=K_c$, while it continuously changes 
for Gaussian case.

\begin{figure}[t]
\begin{center}
\includegraphics[width=120mm]{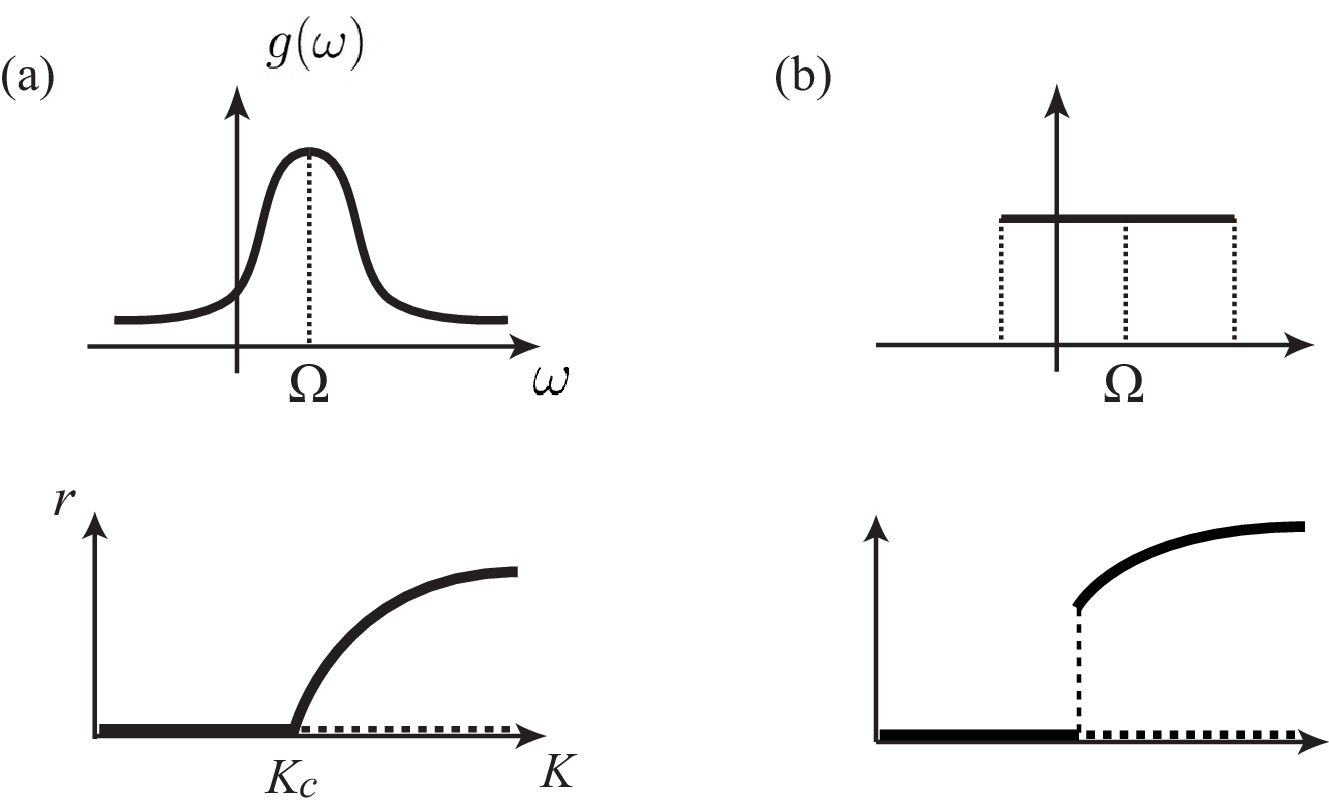}
\caption{
Bifurcation diagrams of the order parameter (lower  panel) for the cases where (a) $g(\omega)$ is a unimodal 
Gaussian (b) $g(\omega)$ is a uniform distribution with the average $\Omega$.
The bifurcation diagrams are independent of $\Omega$ as was mentioned in the previous Remark.
Note that for (b), the order parameter takes a large value just after the bifurcation.} \label{fig:2} 
\end{center}
\end{figure}

\section{Reservoir computing with the Kuramoto model}\label{sec:3}

We introduce the finite-dimensional Kuramoto model with an input term
\begin{equation}\label{KM_input}
    \frac{d \theta_i}{dt} = \alpha u(t) + \omega_i + \frac{K}{N} \sum_{j=1}^N \sin (\theta_j - \theta_i),
\end{equation}
where $i \in [N]$, $\theta_i = \theta_i(\alpha u; t) \in [0,2\pi)$ denotes the phase of an $i$-th oscillator on a circle, 
$\omega_i \in \mathbb R$ denotes its natural frequency, $K>0$ is a coupling strength, 
$N$ is the number of oscillators, $\alpha>0$ is a constant and $u=u(t)$ is an input. 
The case $\alpha=0$ is the (usual) Kuramoto model \eqref{KM}. 
In reservoir computing, we take a sufficiently small $\alpha$.
Indeed, in Goto, Nakajima and Notsu \cite{goto2021twin} using 
the Navier-Stokes equation as a reservoir, they assume that the variance of 
the input is sufficiently small so that the reservoir dynamics is close to the autonomous system.

The order parameter $r(t)$ is similarly  defined by \eqref{Order}. In addition, we define 
the $n$-th order parameter of $\{\theta_j\}_{j=1}^N$ by 
\[
r_n (\alpha u; t) := \frac{1}{N} \sum_{j=1}^N e^{in\theta_j (\alpha u; t)}, \quad n \in \mathbb Z.
\]
Similarly to Section \ref{sec:2}, we consider the continuous limit of \eqref{KM_input} as $N\to\infty$:
with given natural frequency $\omega$ at time $t$, $\rho$ satisfies the continuity equation
\begin{equation}\label{CKM_input}
\begin{cases}
	\dfrac{\partial\rho}{\partial t} + \dfrac{\partial}{\partial \theta} (\rho v) =0,\\
	v =  \alpha u + \omega + K |r| \sin (\psi -\theta),\\
	\displaystyle r(\alpha u; t) = |r(\alpha u; t)| e^{i\psi(\alpha u; t)} 
	:= \int_{\mathbb R} \int_0^{2\pi} e^{i\theta}\rho(\alpha u;\theta, \omega, t) g(\omega)\, d\theta d\omega,
\end{cases}
\end{equation}
where $\rho = \rho(\alpha u; \theta, \omega,t)$ is the density 
  of oscillators at a phase $\theta$ parameterized by a natural frequency $\omega$ at time $t$, 
$v=v(\alpha u;\theta, \omega,t)$ is the drift velocity of the oscillators, and 
$r(\alpha u; t)$ is the order parameter of $\rho$. The $n$-th order parameter is defined as
\begin{equation}\label{n-Order}
r_n(\alpha u; t) := \int_{\mathbb R} \int_0^{2\pi} e^{in\theta}\rho(\alpha u;\theta, \omega, t) g(\omega)\, d\theta d\omega,\quad n \in \mathbb Z.
\end{equation}
We propose reservoir computing with the infinite-dimensional Kuramoto model \eqref{CKM_input} as follows:\\

\noindent{\bf Reservoir computing and its purpose.} 

The reservoir computing with \eqref{CKM_input} consists of mappings $u(t) \mapsto \rho(\alpha u;\theta, \omega, t) \mapsto y(t)$,
where $y(t)\in \mathbb{R}$ is an output function.
The map $u \mapsto \rho$ is defined  by a solution of (\ref{CKM_input}).
Our purpose is to find the mapping (readout) $ \rho(\alpha u;\theta, \omega, t) \mapsto y(t)$ so that 
\[
\|y_{\mathrm{tar}} - y \| \approx 0,
\]
for a given target function $y_{\text{tar}}(t) \in \mathbb R$, where $\|\cdot\|$ is a certain norm or metric. 
For a finite dimensional problem, usually a readout is given by a matrix. However, since \eqref{CKM_input} is an infinite dimensional 
dynamical system, a readout may be a linear operator on an infinite dimensional space.
Thus, we consider the following linear integral operator 
\begin{equation}\label{KRC}
\rho \mapsto y(t) = \mathcal W_{\mathrm{out}} [\rho] (\alpha u, t) := 
\int_{\mathbb R} \int_0^{2\pi} W_{\mathrm{out}}(\theta) \rho(\alpha u;\theta, \omega, t) g(\omega)\, d\theta d\omega,
\end{equation}
where $\mathcal W_{\mathrm{out}}$ is a linear integral operator with an integral kernel $W_{\mathrm{out}} (\theta)$. 
Hence, the purpose of our reservoir computing is to find a suitable function $W_{\mathrm{out}}(\theta )$ satisfying 
$\|y_{\mathrm{tar}} - y \| \approx 0$.
In Theorem 5.2 in \cite{Hart2024}, they consider the similar setting with a view point of randomness.
\\

For mathematical analysis, we always assume that $W_{\mathrm{out}} \in L^2(0,2\pi)$ and it is real-valued in this paper. 
It is important to note that our reservoir computing \eqref{KRC} is convenient because it is expressed by the $n$-th order parameters as follows. 
The integral kernel $W_{\mathrm{out}} \in L^2(0,2\pi)$ has a Fourier expansion
\[
W_{\mathrm{out}}(\theta) = \sum_{n =-\infty}^\infty w_n e^{in\theta}.
\]
Then, by the definition \eqref{n-Order} of $r_n(\alpha u; t)$, we rewrite $\mathcal W_{\mathrm{out}} [\rho] $ as 
\[
\mathcal W_{\mathrm{out}} [\rho] (\alpha u, t) = \sum_{n=-\infty}^\infty w_n
\int_{\mathbb R} \int_0^{2\pi} e^{in\theta} \rho(\alpha u;\theta, \omega, t) g(\omega)\, d\theta d\omega
= \sum_{n=-\infty}^\infty w_n r_n(\alpha u;t). 
\]
Hence, the output of our reservoir computing can be represented as 
\begin{equation}\label{key_representation}
y(t) = \sum_{n=-\infty}^\infty w_n r_n(\alpha u;t). 
\end{equation}
This is an easily implementable setting and we can compute $W_{\text{out}}(\theta)$ and $y(t)$ numerically by using these formulae.
Before giving its numerical algorithm, we consider mathematical results in the next sections.

\section{Approximation ability}\label{sec:4}
As mentioned in the previous section,
the approximation ability of reservoir computing \eqref{KRC} is essentially determined by the property of $\{r_n(\alpha u;t)\}_n$. 
In this section, we focus on the analysis of $\{r_n(\alpha u;t)\}_n$ and derive the condition when
the estimate $\|y_{\mathrm{tar}} - y \| \approx 0$ holds.
Throughout this section, we assume that $g(\omega )$ is an even and unimodal about its mean value $\Omega $.

\begin{thm}\label{thm:main}
Let $K>0$, $\alpha \in \mathbb R$ and $\Omega$ mean value of $g(\omega)$. 
Suppose that $u(t)$ is bounded and continuous on $\mathbb R$.
\begin{itemize}
\item[\rm (i)] For a steady state solution $\rho_*$ of the system \eqref{CKM_input} with $\Omega = \alpha = 0$, 
its order parameters $r_n =: c_n$ are real numbers for any $n\in \mathbb{Z}$.

\item[\rm (ii)] For the system \eqref{CKM_input}, there exists a solution $\rho$ constructed from $\rho_*$ such that 
the $n$-th order parameters of $\rho$ is expressed as
\begin{equation}\label{r_n:basis}
r_n (\alpha u;t) = c_n e^{in (\Omega t + \alpha \int_0^t u(s)\, ds)},\quad t>0, \,\, c_n\in \mathbb{R}.
\end{equation}
\end{itemize}
\end{thm}

\begin{rem}
A few remarks are in order.
\begin{enumerate}[\rm (a)]
    \item 
The system \eqref{CKM_input} for $\Omega = \alpha = 0$ always has the steady state $\rho_* = 1/(2\pi)$ (de-synchronous state).
In this case, $c_n = 0$ for $n \neq 0$.
Thus, we are interested in a non-trivial steady state.
For example, when $g(\omega)$ is even and unimodal or uniform distribution, 
it has a non-trivial steady state $\rho_*$ after the bifurcation $(K>K_c)$ to the synchronous state as in Fig. \ref{fig:2}.
In this case, $c_n \not = 0$ in general (see Theorem~\ref{thm4.6} below). 
\item 
For the system \eqref{CKM_input} with $\Omega = \alpha = 0$, 
the steady state $\rho_* = 1/(2\pi)$ is stable when $K<K_c$ and the non-trivial steady state $\rho_*$ is 
stable when $K>K_c$ in some weak topology, see \cite{chiba2015} for the details. 
Moreover, 
the solution $\rho$ in (ii) is stable if $\rho_*$ in (i) is stable. 
\end{enumerate}
\end{rem}

\begin{cor}\label{cor:main}
Under the same assumption as in Theorem \ref{thm:main}, if $c_n \not =0$ for any $n$, $\Omega \not =0$ and $|\alpha|$ is sufficiently small, then 
$\{r_n(\alpha u; t )\}_n$ given in \eqref{r_n:basis} is a complete basis of $L^2 (0,T)$ with $T \sim 2\pi/\Omega + O(\alpha)$. 
\end{cor}

Our main results mean that the Kuramoto reservoir has almost the same approximation ability as the Fourier basis 
when $c_n \not =0$ for any $n$, $\Omega \not =0$ and $|\alpha|$ is sufficiently small. 
See Theorem \ref{thm4.6} for a sufficient condition for $c_n \not =0$.

To prove Theorem \ref{thm:main}, we use the following lemma,
which is easily proved by integrating $(\partial / \partial \theta) (\rho v) = 0$. 

\begin{lem}\label{lem:sss}
Let $\Omega=0$ and $\alpha=0$. A steady state solution $\rho_*$ of \eqref{CKM_input} is given as 
\[
\rho_* (\theta,\omega) =
\begin{dcases}
 \delta \left(\theta - \mathrm{Arcsin}\, \left(\frac{\omega}{Kr_*}\right)\right),\quad &|\omega|<Kr_*,\\
\frac{1}{2\pi} \frac{\sqrt{\omega^2 - K^2 r_*^2}}{|\omega- Kr_* \sin \theta |}, &|\omega|>Kr_*,
\end{dcases}
\]
where $\delta$ is the Dirac delta function and 
$r_*$ is the $1$-st order parameter of $\rho_*$.
\end{lem}

\begin{proof}[Proof of Theorem \ref{thm:main}]
In order to prove the statement (ii), it is enough to show the statement (i). In fact, for the system \eqref{CKM_input},
change the variables by 
\[
\begin{dcases}
\theta = \tilde{\theta} + \Omega t + \alpha \int_0^t u(s)\,ds,\\
\omega = \tilde{\omega} + \Omega,\\
g(\tilde{\omega}+\Omega) = \tilde{g}(\tilde{\omega}), \\
\rho (\alpha u; \theta, \omega, t) = \tilde{\rho}(\tilde{\theta}, \tilde{\omega},t).
\end{dcases}
\]
Then, it is easy to verify that the system is reduced to the case $\Omega = \alpha = 0$.
More precisely, 
let $\tilde{r}_n$ be the order parameters for the case $\Omega = \alpha = 0$. Then we can calculate $r_n$ as 
\[
\begin{split}
r_n(\alpha u; t) & = \int_{\mathbb R} \int_0^{2\pi} e^{in\theta}\rho(\alpha u;\theta, \omega, t) g(\omega)\, d\theta d\omega\\
& =  \int_{\mathbb R} \int_0^{2\pi} e^{in\tilde{\theta}} e^{in (\Omega t + \alpha \int_0^t u(s)\, ds)} 
        \tilde{\rho}(\tilde{\theta}, \tilde{\omega}, t) \tilde{g}(\tilde{\omega})\, d\tilde{\theta} d\tilde{\omega}\\
& = \tilde{r}_n(t) e^{in (\Omega t + \alpha \int_0^t u(s)\, ds)},
\end{split}
\]
Therefore, once we prove $\tilde{r}_n(t) =: c_n \in \mathbb{R}$ when $\tilde{\rho} = \rho_*$ is a steady state, we conclude the statement (ii).

Suppose $\Omega=0$ and $\alpha=0$, and let $\rho_*$ be a steady state. 
By Lemma~\ref{lem:sss}, its $n$-th order parameters are expressed as 
\[
\begin{split}
r_n & = 
 \int_{\mathbb R} \int_0^{2\pi} e^{in\theta}\rho_*(\theta, \omega) g(\omega)\, d\theta d\omega\\
 & = 
 \int_{|\omega|< Kr_*} \int_0^{2\pi} 
  e^{in\theta}\delta \left(\theta - \mathrm{Arcsin}\, \left(\frac{\omega}{Kr_*}\right)\right) g(\omega)\, d\theta d\omega\\
 &\qquad\qquad \qquad  + \frac{1}{2\pi} \int_{|\omega|>Kr_*} \int_0^{2\pi} e^{in\theta} 
                           \frac{\sqrt{\omega^2 - K^2 r_*^2}}{|\omega- Kr_* \sin \theta |} g(\omega)\, d\theta d\omega\\
 & =: a_n + b_n =: c_n.
\end{split}
\]

Since $g(\omega)$ is an even function about $0$, 
we write 
\[
\begin{split}
a_n  = 
\int_{|\omega|< Kr_*} e^{in \mathrm{Arcsin}\, \left(\frac{\omega}{Kr_*}\right)} g(\omega)\, d\omega
 =
\int_{|\omega|< Kr_*} e^{-in \mathrm{Arcsin}\, \left(\frac{\omega}{Kr_*}\right)} g(\omega)\, d\omega. 
\end{split}
\]
By Euler's formula, this implies that $a_n$ is a real number for any $n$ as
\begin{equation}
a_n = \int_{|\omega|< Kr_*} \cos \left(n\cdot \mathrm{Arcsin}\, \left(\frac{\omega}{Kr_*}\right)\right) g(\omega)\, d\omega \in \mathbb R.
\label{an}
\end{equation}

As for $b_n$, 
by making the change
$(\theta,\omega) \mapsto (\theta+\pi, -\omega)$, we have
\[
b_n = 
\frac{1}{2\pi} \int_{|\omega|>Kr_*} \int_{-\pi}^{\pi} e^{in\theta} e^{in\pi} A(\omega,\theta) g(\omega)\, d\theta d\omega 
=
\begin{cases}
-b_n\quad &\text{if $n$ is odd},\\
b_n\quad &\text{if $n$ is even},
\end{cases}
\]
where 
\[
A(\omega,\theta):=
\frac{\sqrt{\omega^2 - K^2 r_*^2}}{|\omega- Kr_* \sin \theta |} \in \mathbb R.
\]
This implies $b_n = 0$ if $n$ is odd (this fact will be used later). 
We can also write $b_n$ as
\[
b_n = 
\frac{1}{2\pi} \int_{|\omega|>Kr_*}\left( \int_{-\frac{\pi}{2}}^{\frac{\pi}{2}}
        +  \int_{\frac{\pi}{2}}^{\frac{3\pi}{2}}\right) e^{in\theta} A(\omega,\theta) g(\omega)\, d\theta d\omega ,
\]
and by making the change
$(\theta,\omega) \mapsto (\theta + \pi, -\omega)$ for the second integral, for any even $n$, we have 
\begin{align}
b_n &= 
\frac{1}{\pi}
 \int_{|\omega|>Kr_*}\int_{-\frac{\pi}{2}}^{\frac{\pi}{2}} \cos (n\theta) A(\omega,\theta) g(\omega)\, d\theta d\omega \nonumber \\ 
&= \frac{2}{\pi}
 \int^{\infty}_{Kr_*} \int_{-\frac{\pi}{2}}^{\frac{\pi}{2}} \cos (n\theta) A(\omega,\theta) g(\omega)\, d\theta d\omega \in \mathbb{R}.
 \label{bn}
\end{align}
Thus, $c_n \in \mathbb{R}$ and the proof of Theorem \ref{thm:main} is completed.
\end{proof}

To prove Corollary \ref{cor:main}, we use the following lemma.

\begin{lem}\label{lem:complete}
Let $\{g_k\}_k$ be a family of elements in a Hilbert space $H$ with inner product $(\cdot,\cdot)_H$. Suppose that $f=0$
if $f \in H$ and $(f, g_k)_{H}=0$ for any $k$. Then $\{g_k\}_k$ is complete in $H$.
\end{lem}

\begin{proof}[Proof of Corollary \ref{cor:main}] 
Define $T>0$ by 
\[
2\pi =\Omega T + \alpha \int_0^T u(s)\,ds. 
\]
For any $f \in L^2(0,T)$, we suppose 
\[
\int_0^T f(t) e^{in (\Omega t + \alpha \int_0^t u(s)\, ds)}\, dt =0
\]
for any $n$. 
Put $\tilde{t} = \Omega t + \alpha \int_0^t u(s)\,ds$. 
If $u$ is bounded and continuous, $\Omega \not =0$, and $|\alpha|$ is sufficiently small, then 
the map $\eta:t\mapsto\tilde t$ is one-to-one and of $C^1$.
Hence,
the inverse function theorem gives
\[
\int_0^{2\pi} \frac{f(\eta^{-1}(\tilde t))}{\Omega + \alpha u(\eta^{-1}(\tilde{t}))} e^{in \tilde t}\, d\tilde t =0
\]
for any $n$. Since $\{e^{in \tilde t}\}_n$ is an orthogonal basis of $L^2(0,2\pi)$, this implies 
that 
\[
\frac{f(\eta^{-1}(\tilde t))}{\Omega + \alpha u(\eta^{-1}(\tilde{t}))} =0.
\]
Since $\Omega + \alpha u(\eta^{-1}(\tilde{t})) \neq 0$ if $|\alpha|$ is sufficiently small, we have $f=0$. 
Therefore, $\{e^{in (\Omega t + \alpha \int_0^t u(s)\, ds)}\}_n$ is a complete basis of $L^2(0,T)$
by Lemma \ref{lem:complete}. 
Thus, we conclude Corollary \ref{cor:main}. 
\end{proof}

Now we give a sufficient condition for $c_n \not =0$.
\begin{thm}\label{thm4.6}
Assume that $g(\omega)$ is either uniform distribution or even and unimodal $C^2$ function satisfying $g''(0) < 0$.
Then, there exists $\varepsilon > 0$ such that when $K_c<K<K_c+\varepsilon$, 
the constants $c_n$ given in Theorem \ref{thm:main} are nonzero for any $n \in \mathbb{Z}$.
\end{thm}

\begin{cor}[\textbf{edge of bifurcation}]\label{cor4.7}
Assume that the input $u(t)$ is bounded and continuous, $g(\omega )$ is as Theorem \ref{thm4.6} with the mean $\Omega $.
If $\Omega \neq 0$ and $|\alpha|$ is sufficiently small, 
then there exists $\varepsilon > 0$ such that when $K_c< K<K_c + \varepsilon$, 
the order parameters $\{r_n(\alpha u, t)\}_n$ for the system \eqref{CKM_input} 
gives a complete basis of $L^2 (0,T)$ with $T \sim 2\pi/\Omega + O(\alpha)$. 
\end{cor}

Corollary \ref{cor4.7} immediately follows from Corollary \ref{cor:main} and Theorem \ref{thm4.6}.
As a result, we can express any function $y_{\mathrm{tar}}$ in $L^2(0,T)$ 
as a linear combination of the order parameters $r_n(\alpha u,t)$ of the system \eqref{CKM_input}.
According to \eqref{key_representation}, 
we can find numbers $\{w_n\}_n$ (i.e. readout $W_{\text{out}}(\theta)$) such that 
\begin{equation}
|| y^{(M)} - y_{\mathrm{tar}} ||_{L^2(0,T)} \to 0,\,\, (M \to \infty),
\label{approx}
\end{equation}
where $y^{(M)}(t) = \sum^M_{-M}w_n r_n (\alpha u; t)$ is the partial sum.
\\

In the proof of Theorem \ref{thm4.6}, 
we use the second mean value theorem for definite integrals.

\begin{lem}\label{lem:smvt}
Let $a<b$. Suppose $f : [a,b]\to \mathbb R$ is integrable and $\varphi : [a,b]\to\mathbb R$ is monotone and bounded. 
Then there exists $\xi \in (a,b)$ such that 
\[
\int_a^b f(x)\varphi(x)\, dx = \varphi(a)\int_a^\xi f(x)\, dx 
+
\varphi(b)\int_\xi^b f(x)\, dx.
\]
\end{lem}

\begin{proof}[Proof of Theorem \ref{thm4.6}] 
We prove the theorem only for the even and unimodal case, since 
the proof for the uniform distribution is similar.
It is trivial that $c_0 (= r_0) =1$ by the definition. 
Since we assume $K_c<K$, the constant $c_1 (=r_*)$ is positive because of \eqref{kuramoto} ($r_* \sim O(\sqrt{K-K_c})$). 
Thus, we assume $n\geq 2$ (note that $c_{n} = c_{-n}$).
By (\ref{an}) and (\ref{bn}), $c_n$ is given by 
\begin{align*}
c_n &= a_n + b_n \\
&= \int_{|\omega|< Kr_*}  \cos \left(n \cdot\mathrm{Arcsin}\, \left(\frac{\omega}{Kr_*}\right)\right) g(\omega)\, d\omega \\
& \qquad + \frac{2}{\pi}
 \int^{\infty}_{Kr_*} \int_{-\frac{\pi}{2}}^{\frac{\pi}{2}} \cos (n\theta) A(\omega,\theta) g(\omega)\, d\theta d\omega.
\end{align*}
(for the uniform distribution $g(\omega)$, we can calculate this integral and a proof is more easy).

First, $a_n$ is estimated as follows: By $\omega \mapsto Kr_*x$, we have
\[
a_n = 2Kr_* \int^1_0 \cos (n\cdot \mathrm{Arcsin}\, x) g(Kr_*x)dx.
\]
Due to Lemma \ref{lem:smvt}, there exists $0< \xi < 1$ such that 
\[a_n = 2Kr_* g(0) \int^\xi_0 \cos (n\cdot \mathrm{Arcsin}\, x)dx
 + 2Kr_* g(Kr_*)\int^1_\xi \cos (n\cdot \mathrm{Arcsin}\, x) dx.
\]
We estimate it around the bifurcation point ($K \sim K_c$, \,$Kr_* \sim 0$ and $\xi \sim 1$) as
\begin{align*}
a_n &= 2Kr_* g(0) \int^1_0 \cos (n\cdot \mathrm{Arcsin}\, x) dx
       -2Kr_* g(0) \int^1_\xi \cos (n\cdot \mathrm{Arcsin}\, x) dx \\
& \qquad  + 2Kr_* \left( g(0) + \frac{1}{2}(Kr_*)^2 g''(0) + o((Kr_*)^2) \right)  \int^1_{\xi}\cos (n\cdot \mathrm{Arcsin}\, x) dx \\
& = 2Kr_* g(0) \int^1_0 \cos (n\cdot \mathrm{Arcsin}\, x) dx \\
& \qquad + (Kr_*)^3 (g''(0)+o(Kr_*)) \cdot \int^1_\xi \cos (n\cdot \mathrm{Arcsin}\, x)dx,
\end{align*}
where we used $g'(0) = 0$ because $g$ is an even function.
The first integral is calculated as
\begin{align}
\int^1_0 \cos (n\cdot \mathrm{Arcsin}\, x) dx = -\frac{\cos (n\pi/2)}{n^2-1}.
\label{an1}
\end{align}
For the second one, putting $x=\sin y,\, \xi = \sin \eta$ gives
\begin{align}
F_n(\eta) & := \int^1_\xi \cos (n\cdot \mathrm{Arcsin}\, x)dx
 = \int^{\pi/2}_\eta \cos (ny) \, \cos y \, dy, \nonumber \\
 &= \frac{1}{1-n^2}\left( \cos ((\pi n)/2) - \cos (n\eta)\sin (\eta) + n\cdot \sin (n\eta) \cos (\eta) \right).
\label{an2}
\end{align}
The decay rate $F_n(\eta) \sim O(n)$ as $n\to \infty$ will be used in Section \ref{sec:prediction}.
\vskip.5\baselineskip

\textbf{(I)} Suppose that $n$ is an odd number.
Then, the above value (\ref{an1}) is zero.
Further, we have proved in the proof of Theorem \ref{thm:main} that $b_n = 0$.
Thus, we have
\[
c_n = a_n + b_n = (Kr_*)^3 (g''(0)+o(Kr_*)) \cdot F_n(\eta).
\]
Note that $\xi = 1$ corresponds to $\eta = \pi/2$.
 Then, we can verify that $F_n(\pi/2) = F_n'(\pi/2)  =F_n''(\pi/2) =0$ and
 \[F_n'''(\pi/2) = -2n \sin (\pi n/2) \neq 0 \,\, (n: \text{odd}).
 \]
 This proves that $F_n(\eta) \neq 0$ for small $Kr_*>0$ and there exists $\varepsilon$ such that $c_n \neq 0$ for $K_c < K < K_c + \varepsilon$.
 Note that $\varepsilon$ is independent of $n$ because so is the factor $g''(0)+o(Kr_*)$. 
\vskip.5\baselineskip

\textbf{(II)} Suppose $n$ is an even number. 
In this case, $b_n \neq 0$ and 
\begin{align}
c_n &=  -2Kr_* g(0)\frac{\cos (n\pi/2)}{n^2-1} + (Kr_*)^3 (g''(0)+o(Kr_*)) \cdot \int^1_\xi \cos (n\cdot \mathrm{Arcsin}\, x)dx \nonumber \\
&\qquad + \frac{2}{\pi}
 \int^{\infty}_{Kr_*} \int_{-\frac{\pi}{2}}^{\frac{\pi}{2}} \cos (n\theta) A(\omega,\theta) g(\omega)\, d\theta d\omega.
 \label{cn}
\end{align}
Let us estimate the third term $b_n$.
We apply Lemma \ref{lem:smvt} with 
\[
f(\theta) := 
 \cos (n\theta)
\quad \text{and}\quad 
\varphi(\theta) := \frac{2}{\pi}
\int^{\infty}_{Kr_*}A(\omega,\theta) g(\omega)\, d\omega.
\]
Let us show that $\varphi (\theta )$ is monotonically increasing.
It is easy to see that the integral exists and $\partial A /\partial \theta >  0$
 for $\omega > Kr_*$ and $-\pi/2 < \theta < \pi /2$.
 This yields
 \[ \varphi (\theta_1) - \varphi (\theta_2)
  = \frac{2}{\pi}
 \int^{\infty}_{Kr_*} (A(\omega,\theta_1) - A(\omega,\theta_2)) g(\omega)\, d\omega > 0 \]
for any $-\pi/2 \leq \theta_2 < \theta_1 \leq \pi /2$.
Hence, Lemma \ref{lem:smvt} is applicable to obtain
\begin{align*}
b_n & =   \varphi\left(-\frac{\pi}{2}\right)
\int_{-\frac{\pi}{2}}^\xi \cos (n\theta)\, d\theta
+
\varphi\left(\frac{\pi}{2}\right)
\int_\xi^{\frac{\pi}{2}} \cos (n\theta)\, d\theta \label{bn2}\\
& = 
\frac{1}{n} \sin (n\xi) \left(
\varphi\left(-\frac{\pi}{2}\right) - \varphi\left(\frac{\pi}{2}\right)
\right). \nonumber
\end{align*}
To calculate it, note that 
\[
A\left(\omega , -\frac{\pi}{2}\right) - A\left(\omega , \frac{\pi}{2}\right)
= \frac{\sqrt{\omega ^2- K^2r_*^2}}{\omega +Kr_*} - \frac{\sqrt{\omega ^2- K^2r_*^2}}{\omega -Kr_*}
=- \frac{2Kr_*}{\sqrt{\omega ^2-K^2r_*^2}}
\]
when $\omega > Kr_*$.
This provides
\begin{align*}
& \varphi\left(-\frac{\pi}{2}\right) - \varphi\left(\frac{\pi}{2}\right)
 = -\frac{4Kr_*}{\pi}\int_{Kr_*}^\infty \frac{g(\omega )}{\sqrt{\omega ^2-K^2r_*^2}}d\omega \\
&=  -\frac{4Kr_*}{\pi}\int_{Kr_*}^{Kr_* + \delta} \frac{g(\omega )}{\sqrt{\omega ^2-K^2r_*^2}}d\omega
-\frac{4Kr_*}{\pi}\int_{Kr_* + \delta}^\infty \frac{g(\omega )}{\sqrt{\omega ^2-K^2r_*^2}}d\omega,
\end{align*} 
where $\delta$ is a small positive number.
The second term is of order $O(Kr_*)$ as $Kr_* \to 0$.
For the first term, we can verify that it is nonzero and $O(Kr_* \log (Kr_*))$ as $Kr_* \to 0$. Putting
\[
\varphi\left(-\frac{\pi}{2}\right) -  \varphi\left(\frac{\pi}{2}\right) =B,
\]
we obtain
\begin{equation}
b_n = B \frac{\sin (n \xi)}{n}, \quad B\sim O(Kr_* \log (Kr_*)),\, B\neq 0.
\label{bn2}
\end{equation}
Now we have estimations of all three terms of (\ref{cn}) for small $Kr_*>0$.
 \vskip.5\baselineskip
Case (II-1) If the third term is the leading term as $Kr_* \to 0$, 
$c_n \neq 0$ because $B\neq 0$.
\vskip.5\baselineskip
Case (II-2) If the first term is the leading term as $Kr_* \to 0$ (it may happen,
for instance, accidentally $\sin (n\xi) = 0$), then obviously $c_n \neq 0$.
\vskip.5\baselineskip
Case (II-3) If the second term is the leading term (it happens when the first and the third terms cancel each other).
Then, the proof is the same as the case (I).
For even $n$, we can verify that $F_n(\pi/2) = F'(\pi/2) = 0$ and 
\[F_n'' (\pi/2) = \cos (n\pi/2) \neq 0 \,\, (n: \text{even}).\]
Thus, we obtain $c_n \neq 0$.

This completes the proof of Theorem \ref{thm4.6}. 
\end{proof}

In summary, we have obtained the following results to explain the edge of bifurcation:
\begin{itemize}
  \item We need $K_c < K$ (after bifurcation) because when $K_c > K$, the trivial steady state $r_n \equiv 0 \, (n\neq 0)$ is stable.
  \item In particular, just after the bifurcation $K_c < K < K_c + \varepsilon$, we could prove $r_n \neq 0$.
  \item We need $\Omega \neq 0$ because when $\Omega =0$, Corollary \ref{cor:main} does not hold. 
  This means that we need a bifurcation from the steady state to periodic state 
(if $\Omega = 0$, the bifurcation is from the steady state to the steady state).
  \item Therefore, under the assumptions in Theorem \ref{thm4.6} and Corollary \ref{cor4.7}, 
$\{r_n (\alpha u, t) \}_n$ consists of the $L^2(0,T)$ basis and 
there exists a readout $W_{\mathrm{out}}(\theta )$ satisfying (\ref{approx});\,\,
$\|y- y_{\mathrm{tar}} \| \approx 0$
\end{itemize}

\textbf{Error estimate.}
In the rest of this section, we consider the error estimate for (\ref{approx}).
Define the map $\eta (t) = \Omega t + \alpha \int^{t}_0 u(s)ds$ as in the proof of Corollary \ref{cor:main}.
Since it is $C^1$ and $\eta'(t) \neq 0$ for any $t\in \mathbb{R}$ and small $|\alpha |$, 
there exists the inverse $\eta^{-1}$.
Then, the number $T$ is defined by $2\pi = \eta (T)$.
The error is estimated as
\begin{align*}
& \| y^{(M)} - y_{\mathrm{tar}}\|^2_{L^2(0,T)} \\
&=\int^T_{0} |y^{(M)}(s) - y_{\mathrm{tar}}(s)|^2 ds \\
&= \int^{T}_{0}\left| \sum^M_{-M} w_nc_n \cdot e^{in \eta (s)} - y_{\mathrm{tar}}(s) \right|^2 ds \\
&= \int^{2\pi}_{0} \frac{\left| \sum^M_{-M} w_nc_n \cdot e^{in \tilde{s}} - y_{\mathrm{tar}}(\eta^{-1}(\tilde{s})) \right|^2}
{\Omega +\alpha u(\eta^{-1}(\tilde{s}))} d\tilde{s} \\
& \leq \frac{2\pi}{\Omega }\cdot 
\sup_{0<t<2\pi} \left| \sum^M_{-M} w_nc_n \cdot e^{in t} - y_{\mathrm{tar}}(\eta^{-1}(t)) \right|^2 \cdot (1 + O(\alpha )).
\end{align*}
Thus, the error is obtained by a uniform estimate of the convergence of the Fourier series
$|\sum^M_{-M} w_nc_n \cdot e^{in t} - y_{\mathrm{tar}}(\eta^{-1}(t))|$ 
and it depends on the regularity of $y_{\mathrm{tar}} \circ \eta^{-1}$. 
For example, if $y_{\mathrm{tar}} \circ \eta^{-1}$ is $C^p$ and its $p$-th derivative
is $\beta$-H\"{o}lder continuous, there exists $C>0$ such that
\[
|\sum^M_{-M} w_nc_n \cdot e^{in t} - y_{\mathrm{tar}}(\eta^{-1}(t))| < C \frac{\log M}{M^{p + \beta}}
\]
uniformly in $t \in (0,T)$.
In numerical simulation, $M$ may be the number of oscillators $N$ in the Kuramoto model.


\section{Prediction task} \label{sec:prediction}

In this section, we consider a prediction task when the input $u(t)$
is an almost periodic function.

\begin{definition}\label{def5.1}
A function $u(t)$ on $\mathbb{R}$ is called almost periodic if for any $\varepsilon >0$,
the set
\begin{equation}
T(u;\varepsilon ):= \{ \tau(\varepsilon ) \, ; \, |u(t+\tau) - u(t)| < \varepsilon ,\,\, \forall t\in \mathbb{R}\}
\end{equation}
is relatively dense
(that means $\exists L > 0$ such that $[a, a+L] \cap T(u; \varepsilon ) \neq \emptyset$ for $\forall a\in \mathbb{R}$).
\end{definition}
See Fink \cite{fink1974} for the theory of almost periodic functions.
A $T$-periodic function $u$ is almost periodic with 
$T(u; \varepsilon ) \supset \{ kT \, ; \, k\in \mathbb{Z}\}$ for any $\varepsilon >0$.

Our setting for the prediction task is;
\begin{itemize}
  \item Given data : \\
the input function $u(t)$ that is $C^1$, bounded and almost periodic,
\\
the target function defined by $y_{\mathrm{tar}}(t) = u(t + \Delta t)$ with a constant $\Delta t >0$.

  \item Hyper parameters in the Kuramoto reservoir : \\
the number of oscillators $N$ large,\\
$|\alpha |$ sufficiently small and $K \sim K_c$ so that Corollary \ref{cor4.7} holds, \\
$g(\omega )$ with the mean value $\Omega >0$ defined as follows.
\end{itemize}

Define the invertible map $\eta (t) = \Omega t + \alpha \int^t_{0} u(s)ds$.
Take positive $\tau_1(\varepsilon ) \in T(u; \varepsilon )$,
and define $\Omega >0$ so that $\eta (\tau_1) = 2\pi$.
With this $\Omega $, let us define $\tau_2 < \tau_3 < \cdots $ by $\eta (\tau_k) = 2\pi k, \,\, k\in \mathbb{N}$
($``<" $ follows from the fact that $\eta$ is monotonically increasing).

In what follows, we fix $\tau_k (\varepsilon )$ in this way for small $\varepsilon >0$
and assume that the conditions for Cor.\ref{cor4.7} are satisfied.
Define the output
\begin{equation}
y(t) = \sum^\infty_{n=-\infty} w_n c_n \cdot \mathrm{Exp}
\left[ in \left( \Omega t+ \alpha \int^t_{0} u(s)ds \right)\right]
\end{equation}
as before.
Due to Cor.\ref{cor4.7}, we can find the readout $\{ w_n\}_n$ so that
$||y- y_\mathrm{tar}||_{L^2(0, \tau_1)} = 0$ on the interval $(0, \tau_1)$, and let us fix this $\{ w_n\}_n$.
When $u(t)$ is $T$-periodic, it is easy to see that $y(t)$ is also $T$-periodic.
By putting $\tau_1 = T$, we can extend the interval to $(0, kT),\,\, k\in \mathbb{N}$ periodically.
Our purpose here is to extend the interval $(0, \tau_1)$ for the estimate of $y(t) - y_{\mathrm{tar}}(t)$ 
to the larger interval $(0, \tau_k)$ when $u$ is almost periodic.

Unfortunately, it is very difficult to obtain $L^1$ or $L^2$ estimate of $y(t) - y_{\mathrm{tar}}(t)$.
Indeed, $u(s+\tau) - u(s)$ is of $O(\varepsilon )$ for $\tau =\tau (\varepsilon ) \in T(u;\varepsilon )$,
however, $\int^\tau_0 | u(s+\tau) - u(s)| ds$ is not because $\tau (\varepsilon )$ may get larger as 
$\varepsilon $ decreases.
Thus, we consider the pointwise estimate.
Our main theorem is ;

\begin{thm}\label{thm5.2}
For the above setting of the prediction task, for any $\varepsilon >0,\, k\in \mathbb{N}$ and for $0<t<\tau_k$,
there exists a constant $E_k$ such that
\begin{equation}
|y(t) - y_{\mathrm{tar}}(t)| < E_k \varepsilon ^{1/4} \cdot (1 + O(\varepsilon^{1/2}) ),
\end{equation}
as $\varepsilon \to 0$.
\end{thm}

We will prove the theorem step by step with respect to $k \in \mathbb{N}$.
The triangle inequality provides
\begin{align*}
& |y(t+\tau_k) - y_{\mathrm{tar}}(t+\tau_k)| \\
& \quad \leq |y(t+ \tau_k)-y(t)|+|y_{\mathrm{tar}}(t+\tau_k)-y_{\mathrm{tar}}(t)|+|y(t)-y_{\mathrm{tar}}(t)|.
\end{align*}
Hence, the proof is reduced to the estimations of the three terms in the right hand side.
For this purpose, we need a few lemmas.
The next one is known fact about the primitive function of an almost periodic function,
see Fink \cite{fink1974} for the detail.
\begin{lem}\label{lem5.2}
If $u(t)$ is almost periodic,
\begin{itemize}
\item it is expressed as a formal Fourier series 
$u(t) \sim \sum_n a_n e^{i \lambda _n t},\,\, \lambda _n \in \mathbb{R}$;

\item if $\{ \lambda _n\}$ is bounded away from zero 
(i.e. $\exists M$ such that $|\lambda _n| \geq M > 0$ for any $n$),
the the primitive function $U(t) = \int u(s)ds$ exists and is also almost periodic;

\item $U(t)$ is almost periodic if and only if it is bounded on $\mathbb{R}$;

\item for any $\varepsilon >0,\,\, t\in \mathbb{R}$ and $\tau (\varepsilon ) \in T(u; \varepsilon )$, 
there is a number $\delta (\varepsilon ) = c\varepsilon \cdot (1 + O(\varepsilon )),\,\, c\neq 0$ such that 
$|U(t+\tau) - U(t)| < \delta (\varepsilon )$.
This implies $\tau (\varepsilon ) \in T(U; \delta )$ for some $\delta \sim O(\varepsilon )$.
\end{itemize}
\end{lem}

\begin{lem}\label{lem5.3}
Suppose that an almost periodic $C^1$ function $u(t)$ has the almost periodic primitive function $U(t)$.
Then, for any $k\in \mathbb{N}$, there is a constant $b_k>0$ such that
\begin{align*}
B_k(t, \varepsilon ) := |u(t+\tau_k) - u(t) | < b_k \varepsilon \cdot (1 + O(\varepsilon ))
\end{align*}
for any $t\in \mathbb{R}$ as $\varepsilon \to 0$.
Further, there is a number $F_k (\varepsilon ) \sim O(\varepsilon )$ such that
\begin{equation}
\tau_{k+1} - \tau_k = \tau_1 + F_k (\varepsilon ).
\end{equation}
\end{lem}

Since $u(t)$ is bounded, the lemma gives
\begin{equation}
|u(t+\tau_k) - u(t)| < \sup_t |B_k (t, \varepsilon )| =: \varepsilon _k \sim O(\varepsilon ), \quad k\in \mathbb{N},
\end{equation}
that implies $\tau_k \in T(u; \varepsilon _k)$.

\begin{proof}
We prove by induction.
The case $k=1$ is trivial because $\tau_1 \in T(u; \varepsilon )$.
Assume that Lemma is true for some $k$.
Recall $\tau_k$ is defined through
\begin{align*}
2\pi k = \Omega \tau_k + \alpha \int^{\tau_k}_{0} u(s)ds. 
\end{align*}
This gives
\begin{align*}
2\pi &= \Omega (\tau_{k+1} - \tau_k) + \alpha \int^{\tau_{k+1}}_{\tau_{k}} u(s)ds \\
&= \Omega (\tau_{k+1}-\tau_{k}) + \alpha \int^{\tau_{k+1}-\tau_{k}}_{0} u(s)ds
  + \alpha \int^{\tau_{k+1}-\tau_{k}}_0 \left( u(s + \tau_{k}) - u(s) \right) ds \\
&= \eta (\tau_{k+1}-\tau_{k}) + \alpha \int^{\tau_{k+1}-\tau_{k}}_0 \left( u(s + \tau_{k}) - u(s) \right) ds \\
&= \eta (\tau_{k+1}-\tau_{k}) + \alpha \left( U(\tau_{k+1})-U(\tau_{k+1}-\tau_{k}) \right)
 - \alpha \left( U(\tau_{k}) - U(0) \right),
\end{align*}
where $U$ is the primitive function of $u$.
By the inductive assumption, we have $\tau_k \in T(u; \varepsilon _k)$,
and by Lemma \ref{lem5.2}, there is a number $\delta _k (\varepsilon _k)$ such that
\begin{equation}
|U(t+ \tau_k) - U(t)| < \delta _k(\varepsilon _k) \sim O(\varepsilon _k) \sim O(\varepsilon ).
\end{equation}
Thus, we can write $\eta (\tau_{k+1} - \tau_{k}) = 2\pi + O(\varepsilon _k)$.
Since $\eta$ is $C^1$ invertible, this yields
\begin{equation}
\tau_{k+1}-\tau_{k} = \eta^{-1} (2\pi + O(\varepsilon _k)) = \tau_1 + F_k, 
\quad F_k \sim O(\varepsilon _k) \sim O(\varepsilon ).
\end{equation}
Thus,
\begin{align*}
B_{k+1}(t, \varepsilon ) &= |u(t + \tau_{k+1}) - u(t) | \\
& < | u(t+\tau_{k}+F_k+\tau_1) - u(t+\tau_k + F_k) | \\
& \qquad + |u(t+ \tau_k + F_k) - u(t+\tau_k) | + | u(t+\tau_k) - u(t) | \\
&= B_1 + |u(t+ \tau_k + F_k) - u(t+\tau_k)| + B_k.
\end{align*} 
Since $u$ is $C^1$, we have $u(t+ \tau_k + F_k) - u(t+\tau_k) \sim O(F_k) \sim O(\varepsilon )$,
that completes the proof.
\end{proof}

\begin{lem}\label{lem5.4}
For any $k\in \mathbb{N}$, there is a constant $a_k>0$ such that
\begin{equation}
A_k(t, \varepsilon ) := |y(t+\tau_k) - y(t)| < a_k \varepsilon ^{1/4} \cdot (1 + O(\varepsilon^{1/2}) ), 
\end{equation}
for any $t\in \mathbb{R}$ as $\varepsilon \to 0$.
\end{lem}

\begin{proof}
By the definition of $y$ and $\tau_k$,
\begin{align*}
y(t+\tau_k) &= \sum w_nc_n \cdot \mathrm{Exp}\left[ in \left(\Omega \tau_k+\alpha \int^{\tau_k}_{0}u(s)ds \right)
   + in \left(\Omega t +\alpha \int^{t+\tau_k}_{\tau_k}u(s) ds \right) \right] \\
&= \sum w_nc_n \cdot \mathrm{Exp}\left[ in \left(\Omega t +\alpha \int^{t}_{0}u(s+\tau_k ) ds \right) \right],
\end{align*} 
and 
\begin{align}
y(t+\tau_k) - y(t) 
&=  \sum w_nc_n \cdot \mathrm{Exp}\left[ in \left(\Omega t +\alpha \int^{t}_{0}u(s) ds \right) \right] \nonumber \\
& \qquad \times \left(  \mathrm{Exp}\left[in \alpha \int^{t}_{0}( u(s+\tau_k )-u(s)) ds \right] - 1\right) .
\label{eq5.9}
\end{align}
Put $\int^t_0 ( u(s+\tau_k )-u(s)) ds = \delta _k \sim O(\varepsilon )$ as in the proof of Lemma \ref{lem5.3}
and write $\delta _k = \delta $ for simplicity.
Then, we have
\begin{align*}
A_k(t, \varepsilon ) ^2 = |y(t+\tau_k)-y(t)|^2 
 \leq \sum^\infty_{-\infty} |w_n|^2 \cdot \sum^\infty_{-\infty} |c_n|^2 |e^{in\alpha \delta }-1|^2.
\end{align*}
Since $W(\theta ) \in L^2 (0,2\pi)$, $\sum |w_n|^2$ is bounded.

From the proof of Theorem \ref{thm4.6} (in particular from Eqs.(\ref{an2}) and (\ref{bn2})),
we can verify that there is $D_1>0$ such that $|c_n| < D_1/|n|,\,\, n\neq 0$.
Hence, for large $M \in \mathbb{N}$, the second summation is estimated as
\begin{align}
\sum^\infty_{-\infty} |c_n|^2 |e^{in\alpha \delta }-1|^2 
& < D_1^2 \sum_{n\neq 0} \frac{1}{n^2}(2 - e^{in\alpha \delta } - e^{-in\alpha \delta }) \nonumber \\
&< 2D_1^2 \sum^M_{n=1} (2 - e^{in\alpha \delta } - e^{-in\alpha \delta }) + 8D_1^2\sum^\infty_{n=M+1} \frac{1}{n^2}.
\label{D1}
\end{align}
Since
\begin{align*}
\sum^\infty_{n=M+1} \frac{1}{n^2} < \int^\infty_{M} \frac{dx}{x^2} = \frac{1}{M}, 
\end{align*}
putting $M \geq \varepsilon ^{-\gamma}$ for some $\gamma >0$ gives $\sum^\infty_{n=M+1}1/n^2 < \varepsilon ^\gamma $.
For this $M$, by a Taylor expansion around $\alpha \delta =0$, we can show
\begin{align*}
\sum^M_{n=1} e^{in\alpha \delta } = \frac{e^{i\alpha \delta } (1- e^{iM\alpha \delta })}{1-e^{i\alpha \delta }}
 = M \left( 1 + \sum^\infty_{k=1} P_k(M) (\alpha \delta )^k \right),
\end{align*}
where $P_k(M)$ is a certain polynomial of $M$ with degree $k$.
Hence, we obtain
\begin{align*}
2D_1^2 \sum^M_{n=1} (2 - e^{in\alpha \delta } - e^{-in\alpha \delta })
 &= -2D_1^2M \cdot \sum^\infty_{m=1} (P_{2m}(M)+ \overline{P_{2m}(M)}) \cdot (\alpha \delta )^{2m}.
\end{align*}
Here, the summation runs only for $k=2m$ because we can verify that $P_k(M) + \overline{P_k(M)} = 0$
for odd $k$.
With the leading coefficient $p_{2m} \neq 0$ of $P_{2m} + \overline{P_{2m}}$,
it provides 
\begin{align*}
& 2D_1^2 \sum^M_{n=1} (2 - e^{in\alpha \delta } - e^{-in\alpha \delta })
= -2D_1^2M \sum^\infty_{m=1}(p_{2m} + O(1/M))\cdot (M\alpha \delta )^{2m}  \\
 & \qquad = -2D_1^2M^3\alpha ^2\delta ^2 \sum^\infty_{m=1} (p_{2m} + O(1/M)) (M^2\alpha ^2\delta ^2)^{m-1}.
\end{align*}
Take $ M = \varepsilon ^{-\gamma }$ (more precisely, the closest integer to it) and 
recall that $\delta \sim c\varepsilon (1+O(\varepsilon ))$ for some constant $c$ (Lemma \ref{lem5.2}).
This yields 
\begin{align*}
M^3\alpha ^2\delta ^2\sim c^2 \alpha ^2 \varepsilon ^{2-3\gamma } (1+O(\varepsilon )),
 \quad M^2\alpha ^2\delta ^2\sim c^2 \alpha ^2 \varepsilon ^{2-2\gamma } (1+O(\varepsilon )).
\end{align*}
Comparing with the previous estimate $\sum^\infty_{n=M+1}1/n^2 < \varepsilon ^\gamma $, 
it turns out that two terms in (\ref{D1}) has the same order when $\gamma =1/2$, that results in
\begin{align*}
& 2D_1^2 \sum^M_{n=1} (2 - e^{in\alpha \delta } - e^{-in\alpha \delta }) \\
&= -2D_1^2c^2\alpha ^2 \varepsilon ^{1/2} (1+O(\varepsilon ))
   \sum^\infty_{m=1} (p_{2m} + O(\varepsilon ^{1/2})) (c^2 \alpha ^2\varepsilon (1+O(\varepsilon )))^{m-1} \\
 & = D_2 \varepsilon ^{1/2} \cdot (1 + O(\varepsilon^{1/2} ))
\end{align*}
with some constant $D_2$.
Hence, 
\begin{align*}
A_k(t, \varepsilon ) ^2 
 & < \sum^\infty_{-\infty}|w_n|^2 \cdot 
     ( 8D_1^2 \varepsilon ^{1/2} + D_2 \varepsilon ^{1/2} \cdot (1 + O(\varepsilon^{1/2} )) ), \\
A_k (t, \varepsilon ) 
  & < a_k \varepsilon ^{1/4} \cdot ( 1 + O(\varepsilon^{1/2} ) ) 
\end{align*}
for a constant $a_k$.
\end{proof}

Now we are in a position to prove the main theorem in this section.

\begin{proof}[Proof of Theorem \ref{thm5.2}]
The previous lemmas and $y_{\mathrm{tar}} (t) = u(t+\Delta t)$ give
\begin{alignat}{3}
& |y(t+\tau_k) - y_{\mathrm{tar}}(t+\tau_k)| \nonumber \\
& \quad \leq \quad |y(t+ \tau_k)-y(t)| &+ & \,\, |y_{\mathrm{tar}}(t+\tau_k)-y_{\mathrm{tar}}(t)|
    \quad & + & \,\, |y(t)-y_{\mathrm{tar}}(t)| \nonumber\\
& \quad =   \qquad A_k(t, \varepsilon )  &+ &  \qquad B_k(t+\Delta t , \varepsilon ) &+ & \,\, |y(t)-y_{\mathrm{tar}}(t)| \nonumber \\
& \quad \sim \,\, a_k \varepsilon ^{1/4} (1 + O(\varepsilon^{1/2}) ) \quad
   &+ & \qquad b_k \varepsilon (1 + O(\varepsilon) ) &+ & \,\, |y(t)-y_{\mathrm{tar}}(t)|.
\label{5.10}
\end{alignat}

\noindent \textbf{(i)} Suppose $0<t<\tau_1$. 
We know that $|| y^{(M)} - y_{\mathrm{tar}}||_{L^2(0, \tau_1)} \to 0$ as $M\to \infty$,
where $y^{(M)}$ is the partial sum (see (\ref{approx})).
Since $u(t)$ is $C^1$, the convergence is uniform on any open interval in $0<t<\tau_1$ and 
$|y(t)-y_{\mathrm{tar}}(t)| = 0$.
Let us investigate the behavior of $y(t)$ at $t=\tau_1$.
If $y_{\mathrm{tar}}(0) \neq y_{\mathrm{tar}}(\tau_1)$, $y^{(M)}(\tau_1)$ converges to
$(y_{\mathrm{tar}}(0) + y_{\mathrm{tar}}(\tau_1))/2 $ as $M\to \infty$ (Gibbs phenomenon).
However, since $y_{\mathrm{tar}}$ is almost periodic,
\begin{align*}
\lim_{t \to \tau_1 - 0}y(t) = \frac{1}{2}(2y_{\mathrm{tar}}(\tau_1) + O(\varepsilon ) )
 = y_\mathrm{tar} (\tau_1) + O(\varepsilon ),
\end{align*}
and we have the same estimate at $t = 0$.
\vskip.5\baselineskip
\noindent \textbf{(ii)} 
Next, let us consider (\ref{5.10}) for $k=1$ and $0<t< \tau_1$;
\begin{align*}
 |y(t+\tau_1) - y_{\mathrm{tar}}(t+\tau_1)| \sim O(\varepsilon ^{1/4}) + |y(t)-y_{\mathrm{tar}}(t)|.
\end{align*}
Combined with the result of (i), this shows 
$|y(t) - y_\mathrm{tar} (t)| \sim O(\varepsilon ^{1/4})$ for $\tau_1 < t < 2\tau_1$.
For the behavior of $y(\tau_1)$ from the right, we have
\begin{align*}
\lim_{t\to 0+0} |y(t+\tau_1) - y_\mathrm{tar}(t+\tau_1)| 
  \sim O(\varepsilon ^{1/4}) + \lim_{t\to 0+0}|y(t) - y_\mathrm{tar}(t)| \sim O(\varepsilon ^{1/4}).
\end{align*}
This means that even if $y(t)$ is discontinuous at $t = \tau_1$, 
the gap $|y(\tau_1 - 0) - y(\tau_1 + 0)|$ is of $O(\varepsilon ^{1/4})$.
Therefore, $|y(t) - y_\mathrm{tar} (t)| \sim O(\varepsilon ^{1/4})$ for $0<t<2 \tau_1$.
Then, we use (\ref{5.10}) again for $k=1$ and $0<t< 2\tau_1$ to
obtain $|y(t) - y_\mathrm{tar} (t)| \sim O(\varepsilon ^{1/4})$ for $\tau_1 <t<3 \tau_1$.
Since $3\tau_1 = 2\tau_1 + \tau_1 = \tau_2 + \tau_1 + O(\varepsilon ) > \tau_2$ due to Lemma \ref{lem5.3}, 
we have $|y(t) - y_\mathrm{tar} (t)| \sim O(\varepsilon ^{1/4})$ for $0 <t< \tau_2$.
A possible discontinuity of $y(t)$ at $t = 2\tau_1$ does not change the estimate as above.

\vskip.5\baselineskip
\noindent \textbf{(iii)} 
Let us consider (\ref{5.10}) for $k=2$ and $0<t< \tau_2$;
\begin{align*}
 |y(t+\tau_2) - y_{\mathrm{tar}}(t+\tau_2)| \sim O(\varepsilon ^{1/4}) + |y(t)-y_{\mathrm{tar}}(t)|.
\end{align*}
This shows $|y(t) - y_\mathrm{tar} (t)| \sim O(\varepsilon ^{1/4})$ for $0< t < 2\tau_2$.
Using (\ref{5.10}) again for $k=2$ and $0 <t< 2 \tau_2$ yields 
$|y(t) - y_\mathrm{tar} (t)| \sim O(\varepsilon ^{1/4})$ for $\tau_2 < t < 3\tau_2$.
Hence, the estimate holds for the interval to $0<t<\tau_3$ because we can verify that $\tau_3 < 3 \tau_2$ 
with the aid of $\tau_3 = \tau_1 + \tau_2 + O(\varepsilon )$ as above.

Repeating this procedure, it turns out that there exists a number $E_k $
such that $|y(t) - y_{\mathrm{tar}}(t)| \sim 
E_k \varepsilon ^{1/4} (1 + O(\varepsilon^{1/2}) )$ for $0<t<\tau_k$.
\end{proof}


\section{Autonomous run for prediction tasks} \label{sec:autonomous}

So far we have considered the input-aware model in the sense that the Kuramoto reservoir always 
receives the information of the input $u(t)$ through (\ref{KM_input}) and gives the output $y(t)$.
In this section, we investigate the autonomous-run system for prediction tasks;
after the training time to determine the readout $\{ w_n\}_n$, the input is replaced by the output
one step-time before.
The setting is the same as the previous section, 
although $u(t)$ is assumed to be a periodic function for simplicity.
\begin{itemize}
  \item Given data : \\
the input function $u(t)$ that is continuous, bounded and periodic with the period $T$,
\\
the target function defined by $y_{\mathrm{tar}}(t) = u(t + \Delta t)$ with a constant $\Delta t >0$.

  \item Hyper parameters in the Kuramoto reservoir : \\
the number of oscillators $N$ large,\\
$|\alpha |$ sufficiently small and $K \sim K_c$ so that Corollary \ref{cor4.7} holds, \\
$g(\omega )$ with the mean value $\Omega >0$ defined by
\begin{align*}
\Omega = \frac{1}{T} \left( 2\pi + \alpha \int^T_{0} u(s)ds \right).
\end{align*}
\end{itemize}
Recall the output is defined by
\begin{equation}
y(t) = \sum^\infty_{n=-\infty} w_n c_n \cdot \mathrm{Exp}
\left[ in \left( \Omega t+ \alpha \int^t_{0} u(s)ds \right)\right].
\end{equation}
By using the above $\Omega $, it is easy to see that $y(t)$ is also $T$-periodic.
By Corollary \ref{cor4.7}, we can define the readout $\{ w_n\}_n$ so that 
$||y-y_\mathrm{tar}||_{L^2(0,T)} = 0$, and we can extend it to the interval $(0, kT)$
for any $k \in \mathbb{N}$ periodically (note that we do not need $\tau_2, \tau_3,\cdots $
defined for the almost periodic case).
Let $T_{\mathrm{tr}} \geq T$ be a training time to determine $\{ w_n\}_n$ during $0<t<T_{\mathrm{tr}}$.
After that ($t> T_{\mathrm{tr}}$), the input $u(t)$ is replaced by the new input (feedback) $\overline{u}(t)$
which is defined by using the $1$-step before of the new output $\overline{y}(t)$ as follows.

\begin{definition}\label{def6.1}
The new input $\overline{u}$ and the new output $\overline{y}$ are defined by
\begin{equation}
\overline{y}(t) = \sum^\infty_{n=-\infty} w_n c_n \cdot \mathrm{Exp}
\left[ in \left( \Omega t+ \alpha \int^t_{0} \overline{u}(s)ds \right)\right],
\end{equation}
and
\begin{equation}
\overline{u}(t) = \left\{ \begin{array}{ll}
u(t) & (0<t<T_{\mathrm{tr}})  \\
\overline{y}(t-\Delta t) & (T_{\mathrm{tr}}<t), \\
\end{array} \right.
\end{equation}
respectively, where $\Delta t$ is the same number as given in the above setting.
\end{definition}
This means that for $T_{\mathrm{tr}}<t$, the output is fed back to the reservoir with the delay 
$\Delta t$ instead of $u(t)$.
Now we replace $u(t)$ in the Kuramoto reservoir (\ref{KM_input}) by $\overline{u}(t)$ and show that 
it gives the same result as the input-aware model.
\begin{thm}\label{thm6.2}
For the autonomous run model above, $y(t) = \overline{y}(t)$ on the bounded interval $(0, kT)$
and $||\overline{y} - y_\mathrm{tar}||_{L^2(0, kT)} = 0$ hold 
for any $k\in \mathbb{N}$. 
\end{thm}

\begin{proof}
Since 
\begin{align*}
\int^t_{0} \overline{u}(s)ds = \left\{ \begin{array}{ll}
\displaystyle \int^t_{0}u(s)ds  & (0<t<T_{\mathrm{tr}}), \\
\displaystyle \int^{T_{\mathrm{tr}}}_{0}u(s)ds 
+ \int^t_{T_{\mathrm{tr}}} \overline{y}(s-\Delta t)ds  & (T_{\mathrm{tr}} < t),  \\
\end{array} \right.
\end{align*}
the new output is rewritten as
\begin{align*}
\overline{y}(t) = \left\{ \begin{array}{ll}
y(t) & (0<t<T_{\mathrm{tr}}), \\
\displaystyle \sum^\infty_{-\infty}w_nc_n e^A & (T_{\mathrm{tr}}<t), \\
\end{array} \right.
\end{align*}
where
\begin{equation}
A:= in \left( \Omega t + \alpha \int^{T_\mathrm{tr}}_{0} u(s)ds
  + \alpha \int^t_{T_{\mathrm{tr}}} \overline{y}(s-\Delta t) ds \right).
\end{equation}
Suppose $T_{\mathrm{tr}}<t<T_{\mathrm{tr}} + \Delta t$.
Since $y(s)=\overline{y}(s)$ when $T_{\mathrm{tr}}-\Delta t < s < T_{\mathrm{tr}}$, we have
\begin{align*}
A &= in \left( \Omega t + \alpha \int^{T_\mathrm{tr}}_{0} u(s)ds
  + \alpha \int^{t-\Delta t}_{T_{\mathrm{tr}}-\Delta t} \overline{y}(s) ds \right) \\
&=  in \left( \Omega t + \alpha \int^{T_\mathrm{tr}}_{0} u(s)ds
  + \alpha \int^{t-\Delta t}_{T_{\mathrm{tr}}-\Delta t} y(s) ds \right).
\end{align*}
Since $y_{\mathrm{tar}} (s) = u(s+\Delta t)$, further it is written as
\begin{align*}
A &= in \left( \Omega t + \alpha \int^{T_\mathrm{tr}}_{0} u(s)ds
  + \alpha \int^{t-\Delta t}_{T_{\mathrm{tr}}-\Delta t} 
    \left( u(s+\Delta t) - y_{\mathrm{tar}}(s) + y(s) \right) ds \right) \\
&= in \left( \Omega t + \alpha \int^{t}_{0} u(s)ds
  + \alpha \int^{t-\Delta t}_{T_{\mathrm{tr}}-\Delta t} 
    \left( y(s) - y_{\mathrm{tar}}(s) \right) ds \right).
\end{align*}
Hence, we obtain
\begin{align}
\overline{y}(t) &= \sum^\infty_{-\infty}w_nc_n 
  \mathrm{Exp}\left[ in \left( \Omega t + \alpha \int^{t}_{0} u(s)ds \right)  \right] \nonumber \\
& \qquad\qquad \times \mathrm{Exp} \left[ in\alpha \int^{t-\Delta t}_{T_{\mathrm{tr}}-\Delta t} 
    \left( y(s) - y_{\mathrm{tar}}(s) \right) ds \right]
\label{ybar}
\end{align}
It is known that there exists $D>0$ such that $||f||_{L^1 (U)}< D||f||_{L^2 (U)}$
for any $f\in L^2 (U)$, where $U \subset \mathbb{R}^m$ is bounded and open ($L^p$-$L^q$ estimate).
Since we know that $||y - y_\mathrm{tar}||_{L^2 (0, kT)} = 0$, the second 
$\mathrm{Exp} \left[ \cdots  \right]$ disappears and we obtain $\overline{y}(t) = y(t)$ on $T_{\mathrm{tr}}<t<T_{\mathrm{tr}}+\Delta t$.

Suppose $T_{\mathrm{tr}} + \Delta t <t <T_{\mathrm{tr}} + 2\Delta t$.
Since $y(s)=\overline{y}(s)$ when $T_{\mathrm{tr}}-\Delta t < s < T_{\mathrm{tr}} + \Delta t$, we have
\begin{align*}
A &=  in \left( \Omega t + \alpha \int^{T_\mathrm{tr}}_{0} u(s)ds
  + \alpha \int^{t-\Delta t}_{T_{\mathrm{tr}}-\Delta t} y(s) ds \right).
\end{align*}
By the same way as above, we obtain
$\overline{y}(t) = y(t)$ on $T_{\mathrm{tr}}<t<T_{\mathrm{tr}}+2 \Delta t$.
Repeating this procedure proves $\overline{y}(t) = y(t)$ on the bounded interval $(0, kT)$.
\end{proof}

\textbf{Error estimate.}
In numerical simulation, the only partial sum is available.
Let
\begin{equation}
y^{(N)}(t) = \sum^N_{n=-N} w_n c_n \cdot \mathrm{Exp}
\left[ in \left( \Omega t+ \alpha \int^t_{0} u(s)ds \right)\right]
\end{equation}
be the partial sum of $y(t)$ and similar for $\overline{y}^{(N)}$.
By the same calculation to obtain (\ref{ybar}), it turns out that they satisfy
\begin{align*}
y^{(N)}(t) - \overline{y}^{(N)}(t) &= \sum^N_{-N}w_nc_n 
  \mathrm{Exp}\left[ in \left( \Omega t + \alpha \int^{t}_{0} u(s)ds \right)  \right] \nonumber \\
& \qquad\qquad \times \left( 1 - \mathrm{Exp} \left[ in\alpha \int^{t-\Delta t}_{T_{\mathrm{tr}}-\Delta t} 
    \left( y^{(N)}(s) - y_{\mathrm{tar}}(s) \right) ds \right) \right].
\end{align*}
Now we are in the same situation as (\ref{eq5.9}) in the proof of Lemma \ref{lem5.4}.
\\
Write $\int^{t-\Delta t}_{T_{\mathrm{tr}}-\Delta t} \left( y^{(N)}(s) - y_{\mathrm{tar}}(s) \right) ds
 = H(N)$, where $H(N) \to 0$ as $N\to \infty$, and the decay rate depends on the regularity
of $y_\mathrm{tar}(t) = u(t+\Delta t)$ as was discussed in the end of Section \ref{sec:4}.
Then, we can conclude that the error is $|y^{(N)}(t) - \overline{y}^{(N)}(t)| \sim H(N)^{1/4}$
on the interval $(0, kT)$.
Note that it may grow as $k$ increases because of the step by step procedure in the above proof.

\section{Experiments}\label{sec:5}
Based on the above argument, we demonstrated the performance around the edge of bifurcation 
in the Kuramoto reservoir computing model using numerical simulation. 
A prediction task was used to evaluate the performance.


Since the $n$-th order parameter $r_n(\alpha u;t)$ can be numerically computed, 
we can obtain an approximate vector $\hat{w}_{\mathrm{out}}$ of $\{w_n\}_{n=0}^\infty$
in the following steps: 
it is an implementable model by the following procedure based on  \eqref{key_representation}.

\medskip 

\noindent\hrulefill

\noindent{\bf Algorithm.} (Kuramoto reservoir computing) \vspace{-2mm}

\noindent\hrulefill

\noindent{\bf Require:} Data set (input/target) $\{u(t_i), y_{\mathrm{tar}}(t_i)\}_{i=1}^T$ and hyperparameter $N \in \mathbb N$.
\begin{enumerate}
\item For training time $t=t_1,t_2,\ldots, t_T$, solve 
\[
\begin{split}
(y_{\mathrm{tar}}(t_1), \cdots, y_{\mathrm{tar}}(t_T))
& =(w_0, w_1,\cdots, w_{2N}) \times \\
&
\begin{pmatrix}
   r_0(\alpha u;t_1) & r_0 (\alpha u;t_2) & \cdots & r_0 (\alpha u;t_T) \\
   \mathrm{Re}(r_1(\alpha u;t_1) )& \mathrm{Re}(r_1 (\alpha u;t_2)) & \cdots & \mathrm{Re}(r_1 (\alpha u;t_T)) \\
      \mathrm{Im}(r_1(\alpha u;t_1) )& \mathrm{Im}(r_1 (\alpha u;t_2)) & \cdots & \mathrm{Im}(r_1 (\alpha u;t_T)) \\
   \vdots & & & \vdots\\
      \mathrm{Re}(r_N(\alpha u;t_1) )& \mathrm{Re}(r_N (\alpha u;t_2)) & \cdots & \mathrm{Re}(r_N (\alpha u;t_T)) \\
   \mathrm{Im}(r_N(\alpha u;t_1)) & \mathrm{Im}(r_N (\alpha u;t_2)) & \cdots & \mathrm{Im}(r_N (\alpha u;t_T)) \\
\end{pmatrix}\\
& =: \hat{w}_{\mathrm{out}} R.
\end{split}
\]
(note that $r_0 = 1$ by the definition).
\item Calculate $\hat{w}_{\mathrm{out}} = (y_{\mathrm{tar}}(t_1), \cdots, y_{\mathrm{tar}}(t_T)) R^{-1}$ (Moore-Penrose inverse).

\item {\bf Return} $y(t) = \hat{w}_{\mathrm{out}} 
(r_0(\alpha u;t), \mathrm{Re}(r_1(\alpha u;t)), \mathrm{Im}(r_1(\alpha u;t)), \ldots, \mathrm{Im}(r_N(\alpha u;t)))^\top$.
\end{enumerate} \vspace{-2mm}
\noindent\hrulefill

\medskip

As above, we can indeed obtain the approximation $\hat{w}_{\mathrm{out}}$ of $\{w_n\}_{n=0}^\infty$ (i.e. the real Fourier series of 
an integral kernel $W_{\mathrm{out}}(\theta)$ in \eqref{KRC})
and use our reservoir computing for tasks such as time series prediction.\\ 

Fig. \ref{fig:3} (a) shows the composition of the Kuramoto reservoir computing system given in \eqref{KRC}. 
$N$ oscillators receive the input $\alpha u(t)$ through \eqref{KM_input}, 
and $y(t)$ is the output by the product of $\hat{w}_{\mathrm{out}}$ and $\{r_n(\alpha u; t)\}$. In the following tasks, the $\hat{w}_{\mathrm{out}}$ was trained to predict $y_{\mathrm{tar}}(t) =u(t+\Delta t)$ in the training period. 
The training period was set at 30 s, and the test period for the output $y(t)$ was set at 30 s. 
These training and tests were performed after 240 s to stabilize the order parameters. 

\subsection{Relaxation oscillator}

In this example, $u(t)$ was derived from the Van der Pol oscillator represented as 
\begin{equation}
\frac{d^2u}{dt^2} - \mu (1 - u^2) \frac{du}{dt} + u = 0,
\label{vanderpol}
\end{equation}
and $\Delta t$ matched the simulation time step ($0.01\,\mathrm{s}$).
The values of these parameters are summarized in Table 1 below.

\begin{table}[h]
\begin{center}
\caption{Parameter settings.}
\begin{tabular}{ |p{3cm}|p{8cm}|p{3cm}|}
 \hline
 Parameter& Description &Value\\
 \hline
 $N$   & the number of oscillators    &500\\
 $\Delta t$  &   delay  & 0.01 s\\
 $\alpha$ & coefficient of the input & 0.01\\
 $u(t)$    & input signal & Eq.\eqref{vanderpol}\\
 $y_{\mathrm{tar}}(t)$ & target signal & $u(t+\Delta t)$\\
 \hline
\end{tabular}
\end{center}
\end{table}

\begin{figure}[t]
\begin{center}
\includegraphics[width=140mm]{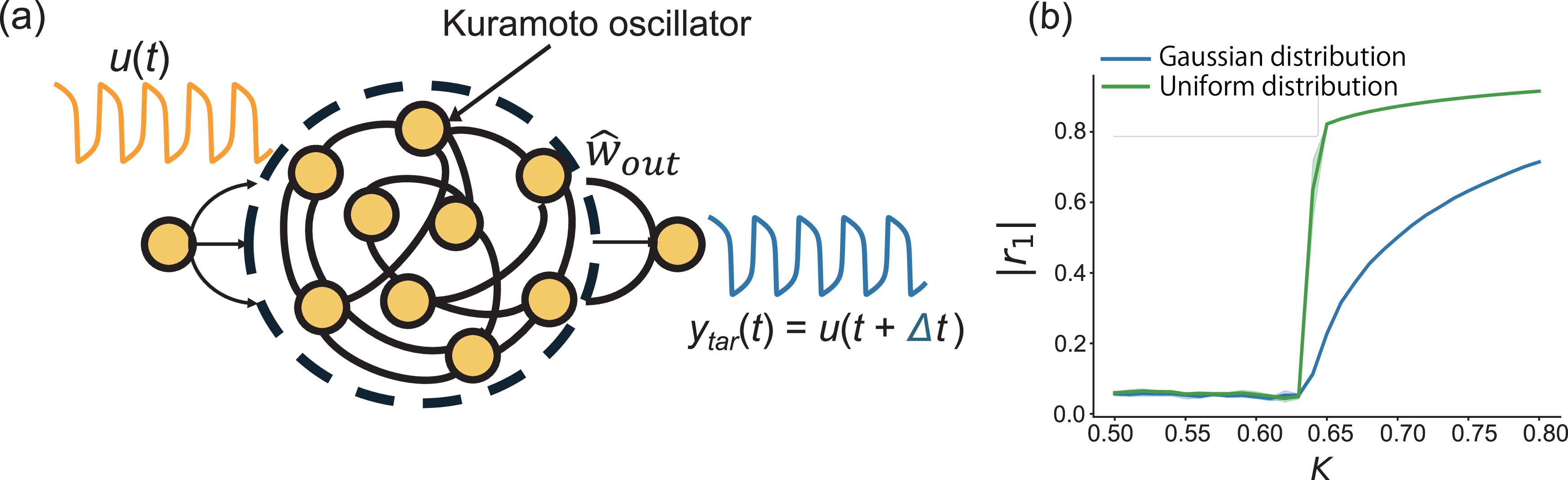}
\vspace*{0.0cm}
\caption{
The Kuramoto reservoir computing system. (a) Schematic diagram of the Kuramoto reservoir computing model and the prediction task. 
(b) The relationship of the order parameter and the bifurcation with coupling strength $K$ 
when $g(\omega)$ is Gaussian distribution or uniform distribution in the numerical simulation. 
The graphs indicate the mean of 10 trials; the shade indicates the 95\% confidence interval.} \label{fig:3} 
\end{center}
\end{figure}

The $\omega_i$ of each oscillator in the Kuramoto model was obtained from the distribution $g(\omega)$.
We simulated two cases;
the Gaussian distribution with mean $\Omega = 0.5$ and standard deviation 0.4, 
or uniform distribution on the interval $[0,1]$ (i.e. the mean value was also $\Omega = 0.5$). 
When the parameters were set in this way, the bifurcation points were 
$K_{\mathrm{c}} \sim 0.638$ (Gaussian distribution), $0.637$ (uniform distribution), 
respectively (Fig. \ref{fig:3} (b), see also Fig.~\ref{fig:2}). 
The reservoir computing performance was quantified by the mean squared error represented as
\begin{equation}
\frac{1}{T'} \sum_{i=1}^{T'} \|y_{\mathrm{tar}}(t_i) - y(t_i) \|^2,
\label{mse}
\end{equation}
where $T'$ is the number of the test step.
As a result, the Kuramoto reservoir computing could predict the Van der Pol oscillator one step after the input 
when $K_c<K$ (Fig. \ref{fig:4} (a)). 
The mean squared error was large when $K<K_c$, and the oscillation was not well predicted, 
however, it became smaller when $K$ exceeded $K_c$ (Fig. \ref{fig:4} (b)). 
On the other hand, if $\Omega = 0$, the mean squared error remained large even when $K_c<K$. 

At $K\approx K_c$, the mean squared error of the uniform distribution is smaller than that of the Gaussian.
This is because, for the Gaussian, the order parameter continuously changes at $K=K_c$, 
while for the uniform distribution, it suddenly jumps to the value $r_1 = \pi/4$ (see Fig.~\ref{fig:2} and Fig.~\ref{fig:3} (b)). 

These results suggest that the Kuramoto reservoir can approximate the target function $y_{\mathrm{tar}} \in L^2(0,T)$ right after the bifurcation, which supports our mathematical results and the edge of bifurcation in Section \ref{sec:4}.

\begin{figure}[h]
\begin{center}
\includegraphics[width=140mm]{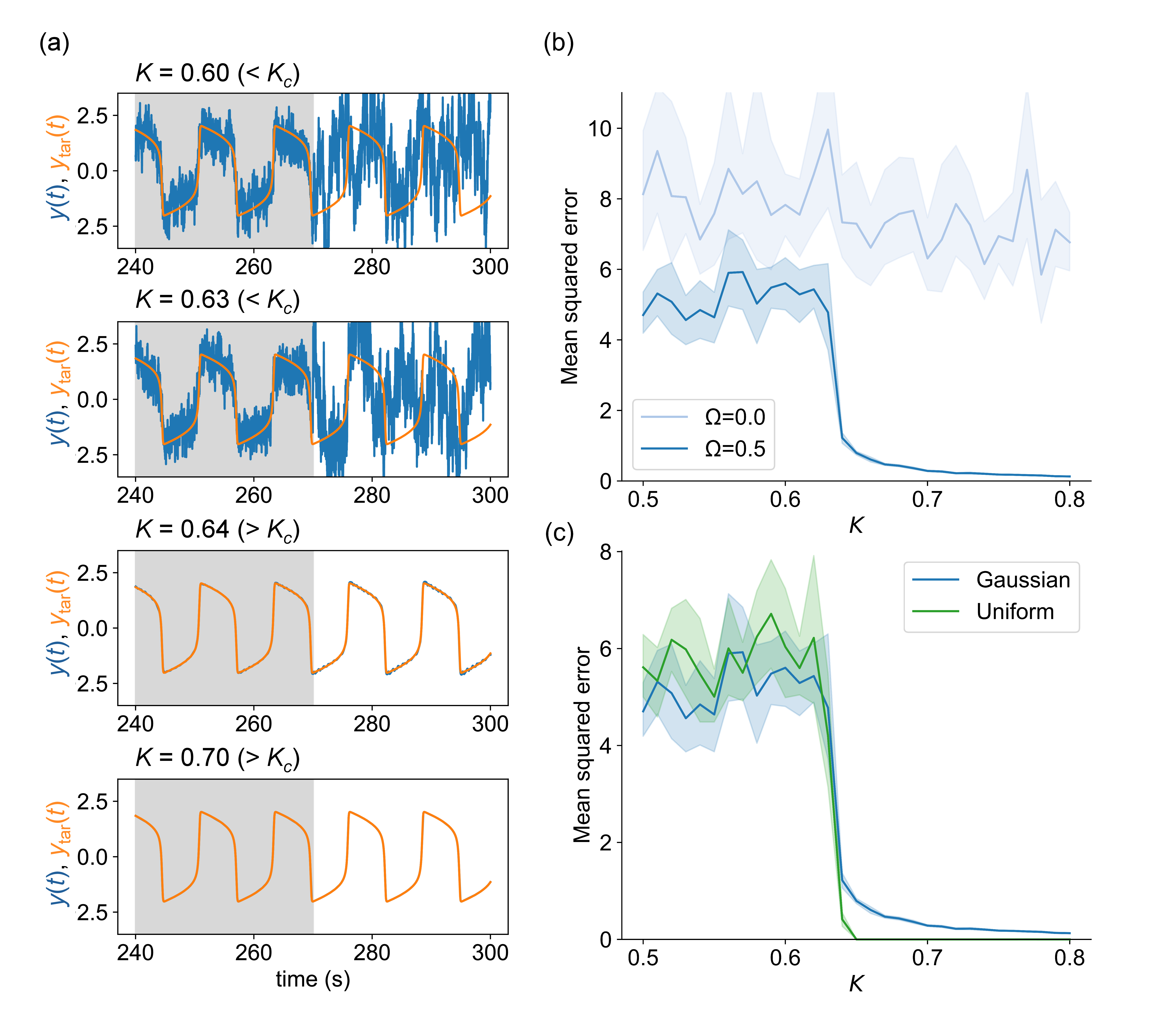}
\caption{
Prediction task using the Van der Pol oscillator. 
(a) The output $y(t)$ (blue line) and the target function $y_{\mathrm{tar}}(t)$ (orange line) when $\Omega=0.5$ with $K=0.60$ ($<K_c$), 
$K=0.63$ ($<K_c$), $K=0.64$ ($>K_c$) and $K=0.70$ ($>K_c$), for which 
we used the Gaussian distribution, and the bifurcation point is $K_c \sim 0.638$. 
The $y_{\mathrm{tar}}(t)$ is Eq.\eqref{vanderpol}, with $\mu \sim 5.66$.
The training period is shown in shading. 
(b) The relationship between the mean squared error and the coupling strength $K$ when $\Omega =0$ and 0.5. 
The solid line represents the mean of 10 trials; the shaded area indicates the 95\% confidence interval. 
(c) Comparison of Gaussian and uniform distribution of $g(\omega)$, with setting $\Omega$ to 0.5.} \label{fig:4} 
\end{center}
\end{figure}

\subsection{Almost periodic function}

The $u(t)$ used here is represented as
\begin{equation}
u(t) = \sin(t) + \sin(\sqrt2 t).
\label{almost}
\end{equation}
This function is almost periodic, since the ratio between the two frequencies $2$ and $\sqrt{2}$ is 
irrational and thus $u(t)$ never repeats exactly.
To find a suitable quasi-period and $\Omega $, suppose that there are integers $m$ and $m'$
such that $2\pi m \approx (2\pi m'/\sqrt{2}) \Rightarrow \sqrt{2}m \approx m'$.
By approximating $\sqrt{2} \approx 7/5$, it is reasonable to set $m=5$ and $m'=7$.
One can regard $u(t)$ as a quasi-periodic signal with an approximate period $\tau \approx 10\pi$ 
for practical analysis. 
Then, the relation $2\pi = \Omega \tau + \alpha \int^\tau_0 u(s)ds$ gives $\Omega \approx 0.2$ 
when $|\alpha |$ is small.
The training and test durations were both set to 32 s, which sufficiently cover the pseudo-period of the signal.

In the prediction task using $\Delta t = 0.01$ s, the Kuramoto reservoir exhibits noisy and incoherent output before bifurcation (Fig.~\ref{fig:5}(a)). In contrast, after the bifurcation, it produces an output that follows the target signal. A sharp drop in the prediction error is observed at the bifurcation point (Fig.~\ref{fig:5}(b)), supporting the theoretical argument in Section~\ref{sec:prediction}.

Interestingly, performance remained stable for a wide range of delays ($\Delta t = 0.1, 1, 10$, and $100$ s) after bifurcation (Fig.~\ref{fig:6}).
In the context of reservoir computing, such robustness to long delays is often discussed in terms of memory capacity.
This property is consistent with recent reports that oscillator-based reservoirs can exhibit 
high memory performance \cite{Kawai2025, Jong2025}.

Finally, we examined the autonomous run, in which the external input was replaced by the reservoir' s own output after the training phase.
Under this closed-loop condition, the reservoir generated a self-sustained oscillation that closely followed the target quasi-periodic trajectory and exhibited a small prediction error after the bifurcation (Fig.~\ref{fig:7}).
This result is consistent with the input-aware model and supports the mathematical theory described in Section~\ref{sec:autonomous}.

\begin{figure}[h]
\begin{center}
\includegraphics[width=140mm]{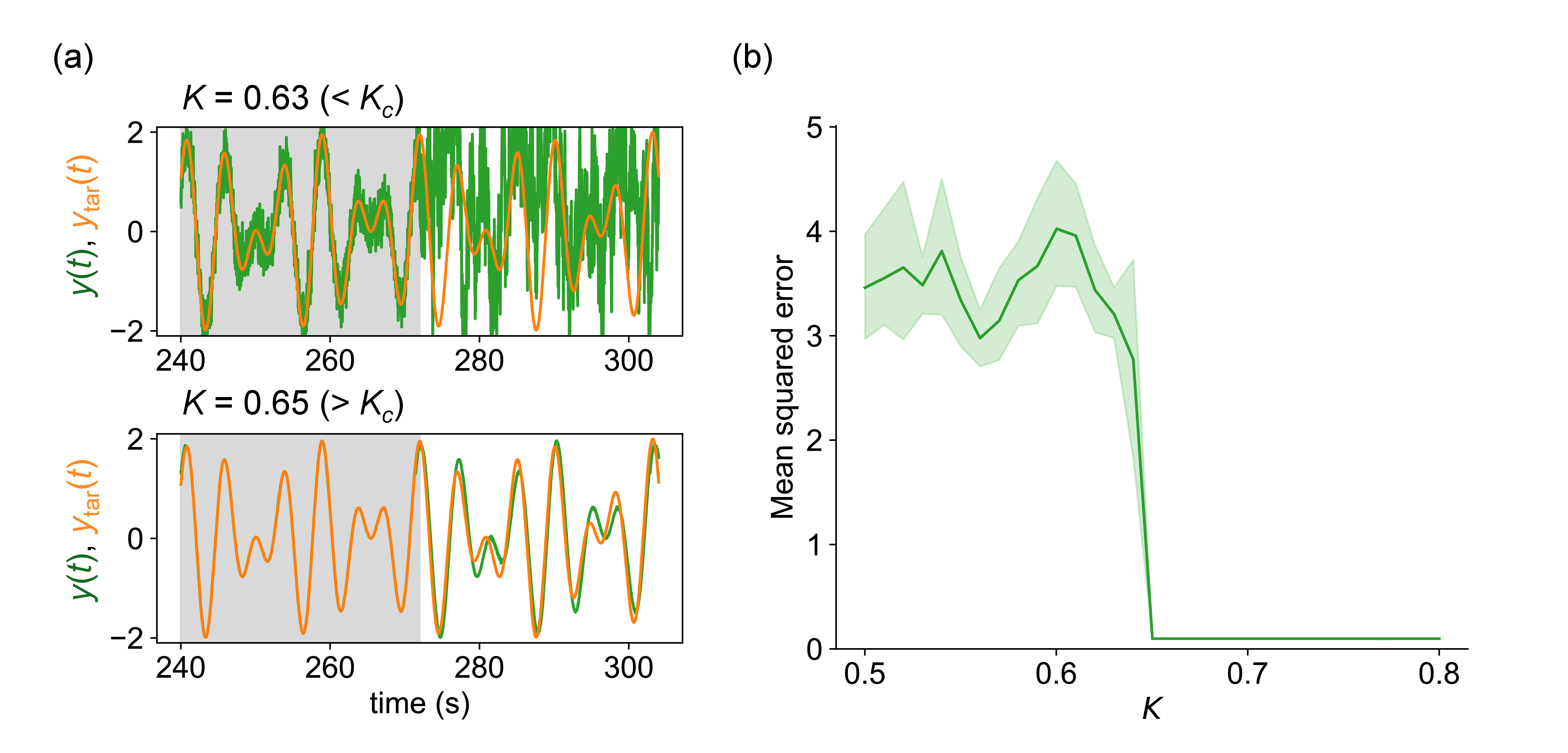}
\caption{
Prediction task using an almost periodic function.
(a) The output $y(t)$ (green line) and the target function $y_{\mathrm{tar}}(t)$ (orange line) with $\Omega=0.2$ for coupling strengths: 
$K=0.63$ ($<K_c$) and 
$K=0.65$ ($>K_c$), where $g(\omega)$ is uniform distribution.
Using Gaussian distribution yields the similar result.
(b) The mean squared error ($n = 10$).} 
\label{fig:5} 
\end{center}
\end{figure}

\begin{figure}[h]
\begin{center}
\includegraphics[width=140mm]{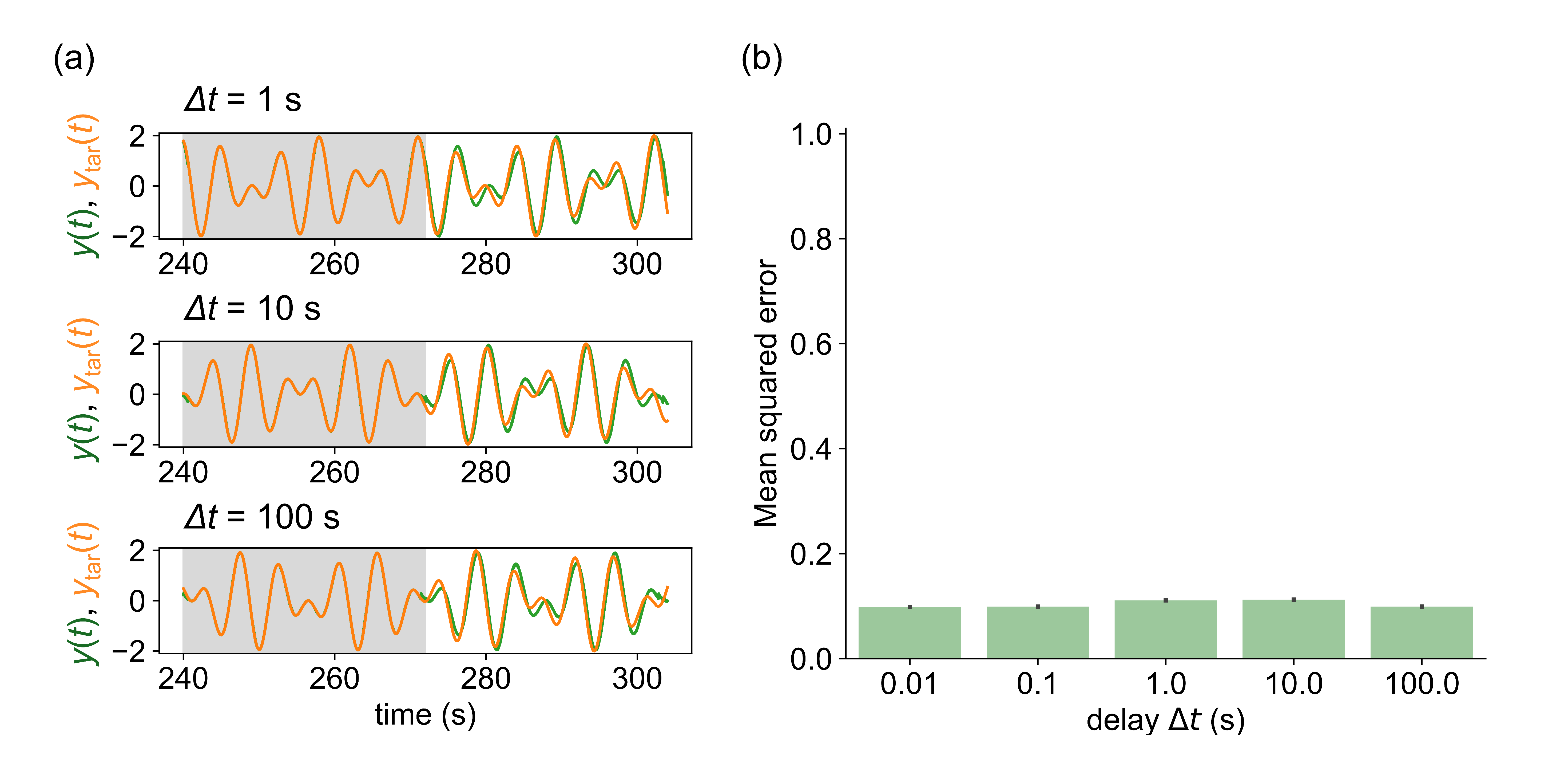}
\caption{
Effect of the delay $\Delta t$ on prediction performance. 
(a) The output $y(t)$ (green line) and the target function $y_{\mathrm{tar}}(t)$ (orange line) with $\Delta t = 1, 10$ and $100$ s. $\Omega=0.2$ and $K=0.65$ ($>K_c$). 
(b) The mean squared error ($n = 10$). The performance remained stable for a wide range of delays.}
\label{fig:6} 
\end{center}
\end{figure}

\begin{figure}[h]
\begin{center}
\includegraphics[width=140mm]{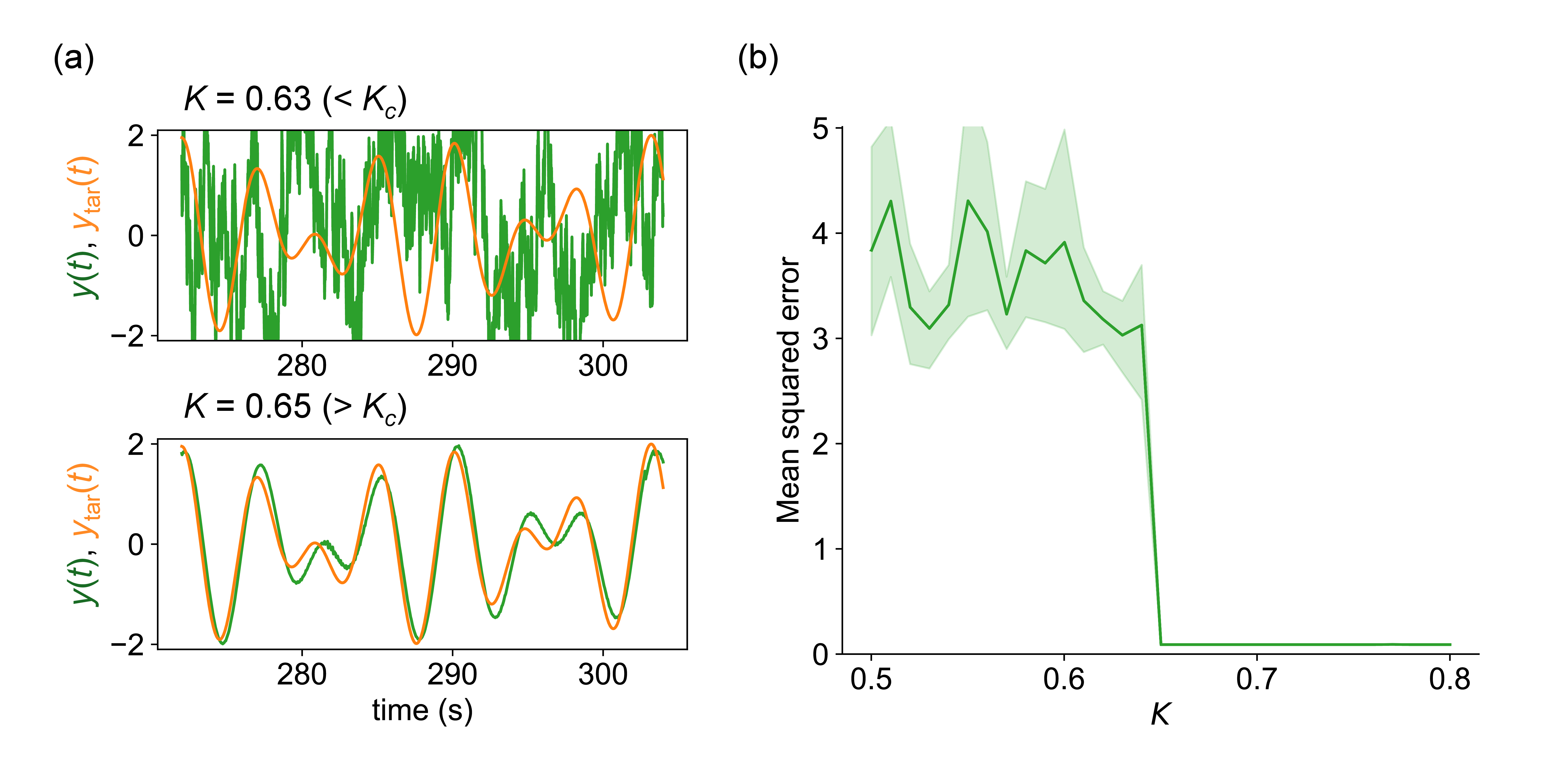}
\caption{
Autonomous run.
(a) The output $y(t)$ (green line) and the target function $y_{\mathrm{tar}}(t)$ (orange line) with $\Omega=0.2$ for coupling strengths: 
$K=0.63$ ($<K_c$) 
and $K=0.65$ ($>K_c$). 
(b) The mean squared error ($n = 10$).}
\label{fig:7} 
\end{center}
\end{figure}

\section*{Acknowledgment}
This work was supported by JST, CREST Grant Number JPMJCR2014, Japan (to HC and KT), JST Moonshot R\&D Program Grant Number 
JPMJMS2023, Japan (to TS).

\bibliographystyle{plain}
\bibliography{Chiba_Taniguchi_Sumi20250925}

\end{document}